\documentclass{amsart}
\usepackage{amscd,amssymb,amsopn,amsmath,amsthm,graphics,amsfonts,accents,enumerate,verbatim,calc}
\usepackage[dvips]{graphicx}
\usepackage[colorlinks=true,linkcolor=red,citecolor=blue]{hyperref}
\usepackage[all]{xy}

\usepackage{tikz}

\usepackage{tikz-cd}

\calclayout

\newcommand{\rt}{\rightarrow}
\newcommand{\lrt}{\longrightarrow}

\newcommand{\st}{\stackrel}

\newcommand{\la}{\lambda}
\newcommand{\La}{\Lambda}

\newcommand{\lan}{\langle}
\newcommand{\ran}{\rangle}
\newcommand{\bu}{\bullet}

\newcommand{\D}{\mathbb{D} }
\newcommand{\K}{\mathbb{K} }
\newcommand{\N}{\mathbb{N} }
\newcommand{\Z}{\mathbb{Z} }

\newcommand{\CA}{\mathcal{A} }

\newcommand{\CE}{\mathcal{E}}
\newcommand{\CF}{\mathcal{F} }

\newcommand{\CI}{\mathcal{I} }

\newcommand{\CK}{\mathcal{K} }
\newcommand{\CL}{\mathcal{L} }
\newcommand{\CM}{\mathcal{M} }
\newcommand{\CN}{\mathcal{N} }
\newcommand{\CP}{\mathcal{P} }

\newcommand{\CS}{\mathcal{S} }
\newcommand{\CT}{\mathcal{T} }
\newcommand{\CX}{\mathcal{X} }

\newcommand{\CV}{\mathcal{V}}
\newcommand{\CU}{\mathcal{U}}

\newcommand{\CB}{\mathcal{B} }

\newcommand{\Mod}{{\rm{Mod\mbox{-}}}}

\newcommand{\mmod}{{\rm{{mod\mbox{-}}}}}

\newcommand{\op}{{\rm{op}}}

\newcommand{\Ker}{{\rm{Ker}}}

\newcommand{\Hom}{{\rm{Hom}}}
\newcommand{\Ext}{{\rm{Ext}}}
\newcommand{\End}{{\rm{End}}}

\theoremstyle{plain}
\newtheorem{theorem}{Theorem}[section]
\newtheorem{corollary}[theorem]{Corollary}
\newtheorem{lemma}[theorem]{Lemma}

\newtheorem{proposition}[theorem]{Proposition}

\theoremstyle{definition}
\newtheorem{definition}[theorem]{Definition}
\newtheorem{example}[theorem]{Example}

\newtheorem{construction}[theorem]{Construction}

\newtheorem{remark}[theorem]{Remark}

\theoremstyle{plain}

\theoremstyle{definition}

\numberwithin{equation}{section}

\begin{document}

\title[different  exact structures  on the monomorphism categories]{different exact structures  on the monomorphism categories}

\author[Rasool Hafezi and Intan Muchtadi-Alamsyah ]{Rasool Hafezi and Intan Muchtadi-Alamsyah}

\address{School of Mathematics, Institute for Research in Fundamental Sciences (IPM), P.O.Box: 19395-5746, Tehran, Iran}
\email{hafezi@ipm.ir}
\address{Algebra Research Group, Faculty of Mathematics and Natural Sciences, Institut Teknologi Bandung, Jalan Ganesha no.10, Bandung, 40132, Indonesia}
\email{ntan@math.itb.ac.id}



\begin{abstract}
Let $\CX$ be a resolving and contravariantly finite subcategory of $\mmod \La$, the category of finitely generated right $\La$-modules. We associate to $\CX$ the subcategory $\CS_{\CX}(\La)$ of the morphism category $\rm{H}(\La)$ consisting of all monomorphisms $(A\st{f}\rt B)$ with $A, B$ and $\rm{Cok} \ f$ in $\CX$. Since $\CS_{\CX}(\La)$ is closed under extensions then it inherits naturally an exact structure from $\rm{H}(\La)$. We will define two other different exact structures else than the canonical one on $\CS_{\CX}(\La)$, and the indecomposable  projective (resp. injective) objects in the corresponding exact categories completely classified. Enhancing $\CS_{\CX}(\La)$ with the new exact structure  provides a framework to construct a triangle functor. Let $\mmod \underline{\CX}$ denote the category of finitely presented functors over the stable category $\underline{\CX}$. We then use the triangle functor to show  a triangle equivalence between the bounded derived category $\mathbb{D}^{\rm{b}}(\mmod \underline{\CX})$ and a Verdier quotient of the bounded derived category of the associated exact category on $\CS_{\CX}(\La)$. Similar consideration is also given for the singularity category  of $\mmod \underline{\CX}$.

\end{abstract}

\maketitle
\section{Introduction}

Let $\La$ be an Artin algebra, and $\mmod \La$ the category of finitely generated right $\La$-modules. The monomorphism (or submodule) category $\CS(\La)$ of $\La$ has as objects all monomorphisms $f$ in $\mmod \La$ and morphisms are given by the commutativity. The representation theory of  submodule categories has been studied intensively by Ringel and Schmidmeier \cite{RS1, RS2}. We can consider $\CS(\La)$ as a subcategory of the morphism category $\rm{H}(\La)$ of $\mmod \La$, that is an abelian category with all morphisms in $\mmod \La$ as whose objects.  For positive integer $n$, let $\La_n=k[x]/(x^n)$, where $k[x]$   is the polynomial ring in one variable $x$ with coefficients in a filed $k$. Let also $\Pi_n$ denote the preprojective algebra of type $\mathbb{A}_n$. The  two functors $F, G: \CS(\La_n)\rt \mmod \Pi_{n-1} $, which one of them has been constructed originally by Auslander and Reiten \cite{AR2}, the other one recently by Li and Zhang \cite{LZ}, were studied in \cite{RZ}.  They were used to show that $\mmod \Pi_{n-1}$ can be obtained from  $\CS(\La_n)$ by factoring out an ideal $\CI$ which is generated by $2n$ indecomposable monomorphisms in the form of $(X\st{1}\rt X)$ and $(0\rt X)$, where $X$
 belongs to the set of all non-isomorphic indecomposable $\La_n$-modules. 
We know that $\La_n$ is a self-injective algebra  of finite representation type and also the stable Auslander algebra of $\La_n$ is isomorphic to $\Pi_{n-1}$ as algebras. Inspired of this observation, these two functors were extended by Eiriksson \cite{E} from the monomorphim category $\CS(\La)$ of a self-injective  algebra  $\La$ of finite representation type  to the module category of the stable Auslander algebra of $\La$, and similar equivalences were proved. We mention recently in \cite{L} a generalization of those results is given. There is another generalization of such results which was given by the first named author \cite{H2} as follows: Let $\CX$ be a subcategory of $\mmod \La$, and let $\CS_{\CX}(\La)$ denote the subcategory of $\CS(\La)$ consisting of all those monomorphisms $(A\st{f}\rt B)$ such that all modules $A, B$ and $\rm{Cok} \ f$ belong to $\CX$. Assume $\CX$ is quasi-resolving, see Subsection \ref{An equivalence} for the definition.   Motivated by one of the two functors worked in \cite{E} and \cite{RZ}, in the setting of functors category, the functor $\Psi_{\CX}$ is defined from $\CS_{\CX}(\La)$ to the category $\mmod \underline{\CX}$ of finitely presented functors over $\underline{\CX}$. The functor $\Psi_{\CX}$ induces a similar equivalence to the one in \cite{E} and \cite{RZ}, in addition, it is proved in \cite[Section 5]{H2} with some further conditions on $\CX$, the functor $\Psi_{\CX}$ makes a nice connection between the Auslander-Reiten theory of $\CS_{\CX}(\La)$ and that of $\mmod \underline{\CX}$. If we assume $\CX$ is resolving, i.e., a quasi-resolving subcategory which is closed under extensions, then the subcategory $\CS_{\CX}(\La)$ so is a resolving subcategory in $\rm{H}(\La)$. Hence, it inherits an exact structure from $\rm{H}(\La)$. On the other hand, we know $\mmod \underline{\CX}$ is an abelian category due to \cite[Proposition 2.11]{MT}. The natural question may arise here whether the functor $\Psi_{\CX}$ is an exact functor with  respect to the above-mentioned exact structures for $\CS_{\CX}(\La)$ and $\mmod \underline{\CX}$. Unfortunately, by considering the  canonical exact structure on $\CS_{\CX}(\La)$, it  is not true the functor $\Psi_{\CX}$ to be exact in general. To resolve this fault, we will introduce a new  exact structure on $\CS_{\CX}(\La)$, denote by $\CS^{\rm{scw}}_{\CX}(\La)$ the new exact category,  which ensures that the functor $\Psi_{\CX}$ is always exact from $\CS^{\rm{scw}}_{\CX}(\La)$ to the  abelian  category $\mmod \underline{\CX}$. The exact functor $\Psi_{\CX}:\CS^{\rm{scw}}_{\CX}(\La)\rt \mmod \underline{\CX}$ establishes a triangle functor form the bounded derived category  $\D^{\rm{b}}(\CS^{\rm{scw}}_{\CX}(\La))$ of bounded complexes over the exact category $\CS^{\rm{scw}}_{\CX}(\La)$ to the bounded derived category $\D^{\rm{b}}(\mmod \underline{\CX})$.  We will use the  induced triangle functor to prove the following theorem.

\begin{theorem}[Theorem \ref{Theorem 6.5} and \ref{Theorem 6.6}]\label{Theorem 1.1}
	Let $\CX$ be a resolving and contravariantly finite subcategory $\CX$ of $\mmod \La.$ Let $\CU$ denote the smallest thick subcategory of $\D^{\rm{b}}(\CS^{\rm{scw}}_{\CX}(\La))$ containing all complexes concentrated in degree zero with terms of the forms $(X\st{1}\rt X)$ and $(0\rt X)$, where $X$ runs through all objects in $\CX$. Then there exists the following equivalences of triangulated categories
$$\D^{\rm{b}}(\CS^{\rm{scw}}_{\CX}(\La))/\CU \simeq \D^b(\mmod \underline{\CX}).$$	 
$$\D_{\rm{sg}}(\CS^{\rm{scw}}_{\CX}(\La))\simeq \D_{\rm{sg}}(\mmod \underline{\CX}).$$ 
\end{theorem}
We recall that the singularity category of an  exact category $\mathcal{C}$ is by definition the Verdier quotient
$\D_{\rm{sg}}(\mathcal{C})=\frac{\D^{\rm{b}}(\mathcal{C})}{\K^{\rm{b}}(\CP)}$, where $\K^{\rm{b}}(\CP)$ is the homotopy category of bonded complexes of projective objects in the exact category $\mathcal{C}$.\\

The above theorem can be thought of as a derived and singularity version of the equivalence given in \cite{E} and \cite{RZ}.\\
In addition, we will define another exact structure on $\CS_{\CX}(\La)$ which gives us an exact category, denote by $\CS^{\rm{cw}}_{\CX}(\La)$, in view of homological property,  is very close to the hereditary algebras.
 Enhancing the additive category $\CS_{\CX}(\La)$ with different exact structures allows us to use better the existing theorems in the setting of exact categories (or more general extriangulated categories recently defined in \cite{NP})  to get some new results for $\CS_{\CX}(\La)$, as we will do in Section 5.\\
 The paper is organized as follows. In Section 2, for later use we state some backgrounds  on the functors category, exact categories and recalling some results from \cite{H2}. Section 3 is a central section of our paper. In this section, we will introduce three different types of exact structures on $\CS_{\CX}(\La)$. Also, projective indecomposable objects and injective indecomposable objects in the relevant exact categories completely classified. In Section 4, we will revisit the functor $\Psi_{\CX}$ relative to the exact structures defined on $\CS_{\CX}(\La)$, and some new equivalences in the level of the stable categories are proved. In Section 5, we will show $\CS_{\CX}(\La)$ has almost split sequences which is, indeed, an application of our consideration to define other exact structures on $\CS_{\CX}(\La)$ different from the canonical one. In Section 6, Theorem \ref{Theorem 1.1}
is proved, and finally an interesting application for singular equivalences of Morita type is given.

{\bf Conversion.}
Throughout this paper, unless otherwise specified, we use the following convention: All subcategories are assumed to be strictly full (i.e., full and closed under isomorphism). $\La$ denotes an Artin algebra over a commutative artinian ring $k$. By $\mmod \La$, resp. $\rm{prj}\mbox{-}\La$,  denotes the category of right finitely generated, resp. projective,  $\La$-modules. All $\La$-modules in this paper are assumed to be finitely generated. By module we mean an $\La$-module. We denote by $D$ the ordinary duality between $\mmod \La$ and $\mmod \La^{\rm{op}}$.

\section{Preliminaries}
\subsection{Functors category}
Let $\mathcal{C}$ be an additive category.  We denote by $\Hom_{\mathcal{C}}(X, Y)$ the set of morphisms from $X$ to $Y.$  For special case $\mathcal{C}=\mmod \La$, we use as usual  $\Hom_{\La}(X, Y)$ in stead of $\Hom_{\mmod \La}(X, Y)$.
An (right) $\mathcal{C}$-module is a contravariant additive functor from $\mathcal{C}$ to the category of abelian groups. If we assume that $\mathcal{C}$ is an essentially small category, then we can define an abelian category $\Mod\mathcal{C}$ of $\mathcal{C}$-modules and natural transformations as morphisms between them.   We call a $\mathcal{C}$-module $F$ \emph{finitely presented} if there exists an exact sequence $\Hom_{\mathcal{C}}(-, X) \st{f} \rt \Hom_{\mathcal{C}}(-, Y) \rt F \rt 0.$ 
We denote by $\mmod \mathcal{C}$ the category of finitely presented $\mathcal{C}$-modules. The {\it Yoneda functor}, defined by sending an objects $C$ in $\mathcal{C}$ to the representable functor $\Hom_{\mathcal{C}}(-, C)$, establishes a fully faithful functor from $\mathcal{C}$ to $\mmod \mathcal{C}$. If $\mathcal{C}$ has split idempotents, the essential image of the Yoneda functor consists of exactly the projective objects in $\mmod \mathcal{C}$. We recall $\mathcal{C}$ has {split idempotents} if every idempotent endomorphism $e=e^2$ of an object $C$ in $\mathcal{C}$ splits, that is, there exists a factorization $C \st{a}\rt B \st{b}\rt C$ of $e$ with $a\circ b=1_B.$ Let $\CX$ be a subcategory of $\mmod \La$ containing $\rm{prj}\mbox{-}\La.$ In this paper, we mostly are dealing with $\mmod \CX$ and $\mmod \underline{\CX}$. Here $\underline{\CX}$ denotes  the stable category of $\mathcal{X}$. The objects of $\underline{\mathcal{X}}$ are the same as the objects of $\mathcal{X}$, which we usually denote by  $\underline{X}$ when an object $X \in \mathcal{X}$ considered as an object in the stable category $\underline{\CX}$, and  the morphisms are given by $\Hom_{\underline{\mathcal{X}}}(\underline{X}, \underline{Y})= \Hom_{\La}(X, Y)/ \CP(X, Y)$, or simply $\underline{\rm{mod}}_{\La}(X, Y)$, where $\CP(X, Y)$ is the subgroup of $\Hom_{\La}(X, Y)$  consisting of those morphisms from $X $ to $Y$ in $\mmod \La$ which factor through a projective object in $\rm{prj}\mbox{-}\La.$ In order to simplify, we will use $(-, X)$, resp. $(-, \underline{X})$, to present the representable functor $\Hom_{\mathcal{X}}(-, X)$, that is only the restriction of the representable functor $\Hom_{\La}(-, X)$ in $\mmod \mmod \La$ on $\CX$,  resp. $\Hom_{\underline{\mathcal{X}}}(-, \underline{X})$, that is similarly the restriction of the representable functor $\underline{\Hom}_{\La}(-, \underline{X})$ in $\mmod \underline{\rm{mod}}\mbox{-}\La$ on $\underline{\CX}$,    in $\mmod \mathcal{X}$, resp. $\mmod \underline{\mathcal{X}}.$ It is well-known that the canonical functor $\pi: \mathcal{X}\rt \underline{\mathcal{X}}$ induces a fully faithful functor  functor $\pi^*:\mmod \underline{\mathcal{X}}\rt \mmod \mathcal{X}$. Hence due to this embedding we can  identify the functors in $\mmod \underline{\mathcal{X}}$ as functors in $\mmod \mathcal{X}$ vanishing on $\rm{prj}\mbox{-}\La.$ In general, $\mmod \CX$ or $\mmod \underline{\CX}$ may not be an abelian category. But, the case for $\mmod \CX$ happens whenever $\CX$ is {\it contravariantly finite}. We say that $\CX$ is contravariantly finite, if for any module $M$ in $\mmod \La$ there is morphism $f:X\rt M$ with $X$ in $\CX$ such that the induced morphism $(-, f)\mid_{\CX}:\Hom_{\La}(-, X)\mid_{\CX}\rt \Hom_{\La}(-, M)\mid_{\CX}$ is onto in $\mmod \CX.$ The symbol ``$\mid_{\CX}$'' is used to show the functor in $\mmod \CX$ obtained by the restriction on $\CX.$ Dually, one can define the notion of {\it covariantly finite} subcategory. A subcategory is called {\it functorialy finite} if it is both contravariantly finite and covariantly finite. A sufficient condition on $\CX$ to force $\mmod \underline{\CX}$
to be an abelian category is so-called quasi-resolving. We refer the reader to the beginning of Subsection \ref{An equivalence} for the definition of quasi-resolving subcategory. If we assume $\mmod \CX$ is an abelian category, then it has split idempotents. In addition, since the endomorphism of every functor in $\mmod \CX$ is semi-perfect, then in this case  $\mmod \CX$ due to \cite[Corollary 4.4]{K2} is a Krull-Schmidt category. Then any functor in $\mmod \CX$ has projective cover, and so we can talk about the existence of minimal projective resolution in $\mmod \CX.$ Similar situation will happen if it is assumed $\mmod \underline{\CX}$ to be an abelian category. 
\subsection{Exact categories}
We recall from \cite{B} the definition  of Quillen exact categories .
Let $\mathcal{C}$ be an additive category. A {\em kernel-cokernel pair} $(i,p)$ in $\mathcal{C}$ is a pair of composable morphisms $X\xrightarrow{i} Y\xrightarrow{p} Z$ such that $i$ is a kernel of $p$ and $p$ is a cokernel of $i$. Assume that $\mathcal{E}$ is a class of kernel-cokernel pairs. A kernel-cokernel pair $(i,p)$ in $\mathcal{E}$ is called a {\em short exact sequence}, or {\it conflation} in some context, in $\mathcal{E}$, which is denoted by $0\rightarrow X\xrightarrow{i} Y\xrightarrow{p} Z\rightarrow 0$.   A morphism $p:Y\rightarrow Z$ is called {\em admissible epimorphism}, or {\it deflation}, also called closed under extensions,  if there exists a morphism $i:X\rightarrow Y$ such that $(i,p)\in\mathcal{E}$. {\em Admissible monomorphisms}, or {\it inflations}, are defined dually.

A class of  kernel-cokernel pairs $\mathcal{E}$ is called an {\em exact structure} of $\mathcal{C}$ if $\mathcal{E}$ is closed under isomorphisms and satisfies the following axioms:

(E0) Identity morphisms are admissible epimorphisms.

(E0)$^{\textup{\tiny {op}}}$ Identity morphisms are  admissible monomorphisms.

(E1) The composition of two admissible epimorphisms  is an admissible epimorphism.

(E1)$^{\textup{\tiny {op}}}$ The composition of two admissible monomorphisms is an admissible monomorphism.

(E2) Given a short exact sequence  $0\rightarrow X\xrightarrow{i} Y\xrightarrow{p} Z\rightarrow 0$ in $\mathcal{E}$ and a morphisms $\varphi:X\rightarrow X'$ in $\mathcal{C}$, there exists a commutative diagram
$$\xymatrix
{0\ar[r] & X\ar[r]^{i}\ar[d]^{\varphi} & Y\ar[r]^{p}\ar[d]^{\varphi'} & Z\ar[r] \ar@{=}[d] & 0\\
	0\ar[r] & X'\ar[r]^{i'} & Y' \ar[r]^{p'} & Z\ar[r] & 0
}$$
such that the second row belongs to $\mathcal{E}$. In this case, $(0\rightarrow X\xrightarrow{\left[                                                                                 \begin{smallmatrix}
	i \\
	\varphi \\
	\end{smallmatrix}
	\right]
} Y\oplus X'\xrightarrow{[\varphi',-i']} Y'\rightarrow 0)\in\mathcal{E}$.

(E2)$^{\textup{\tiny {op}}}$  Given a short exact sequence  $0\rightarrow X\xrightarrow{i} Y\xrightarrow{p} Z\rightarrow 0$ in $\mathcal{E}$ and a morphisms $\phi: Z'\rightarrow Z$ in $\mathcal{C}$, there exists a commutative diagram
$$ \xymatrix
{0\ar[r] & X\ar[r]^{i'}\ar@{=}[d] & Y'\ar[r]^{p'} \ar[d]^{\phi'} & Z'\ar[r] \ar[d]^{\phi} & 0\\
	0\ar[r] & X\ar[r]^{i} & Y \ar[r]^{p} & Z\ar[r] & 0
}$$
such that the first row belongs to $\mathcal{E}$. In this case, $(0\rightarrow Y'\xrightarrow{\left[
	\begin{smallmatrix}
	p' \\
	\phi' \\
	\end{smallmatrix}
	\right]
} Z'\oplus Y\xrightarrow{[\phi,-p]} Z\rightarrow 0)\in\mathcal{E}$.

An {\em exact category} is an additive category $\mathcal{C}$ admits an exact structure $\mathcal{E}$, which is denoted by $(\mathcal{C},\mathcal{E})$. If there is no ambiguity, we only write $\mathcal{C}$.

For example, an additive category $\mathcal{C}$ is an exact category with respect to the class of split short exact sequences, which are isomorphic to $0\rightarrow X\xrightarrow{\left[
	\begin{smallmatrix}
	1 \\
	0 \\
	\end{smallmatrix}
	\right]
} X\oplus Y\xrightarrow{[0,1]} Y\rightarrow 0$ for some $X,Y\in\mathcal{C}$. An abelian category $\mathcal{C}$ is an exact category where the exact structure is given by all the kernel-cokernel pairs in $\mathcal{A}$. If $\CX$ is an extension-closed  subcategory of an abelian $\mathcal{A}$, then $\CX$ naturally inherits an exact structure from $\CA$ by taking  all the short  exact  sequences such that  all terms lie in $\CX$. We recall that $\CX$ is extension-closed, if for each short exact sequence $0 \rt X \rt Y \rt Z\rt 0$ in $\mathcal{A}$ with $X$ and $Z$ in $\CX$ the middle term $Y$ also belongs to $\CX.$
\begin{lemma}\label{Lemma2.1} Let $(\mathcal{C}, \CE)$ be an exact category. If $\CE'$ be a subclass of $\CE$ contains the identity morphisms and closed under isomorphisms and  operations given in axioms $(\rm{E}_1), (\rm{E}_2)$, resp.  $(\rm{E}_1)^{\op}, (\rm{E}_2)^{\op}$, then $(\mathcal{C}, \CE')$ is an exact category.	
	\end{lemma}
\begin{proof}
Note that the axioms are closed under the corresponding operations mean as follows: For $\rm{E}_1$, let $f$ and $g$ be two admissible epimorphisms in $\CE$ for which there are short exact sequences $(i, f)$ and $(j, g)$ in $\CE'$. Then for the admissible epimorphism $f\circ g$ there is a short exact sequence $(k, f\circ g)$ in $\CE'.$ For $(\rm{E}_2)$, assume that $(i, p)$ in $\CE'$ and a morphism $\phi:X\rt X'$, then any short exact sequence $(i', p')$ obtained by  the push-out in the axiom $\rm{E}_2$ has to be in $\CE'.$ Accordingly, we can say similarly for the opposite axioms.	Such  assumptions clearly ensure that  $\CE'$ satisfies the axioms.	
	\end{proof}
An object $P$ of an exact category $(\mathcal{C},\mathcal{E})$ is called {\em projective} if for each admissible epimorphism $p:Y\rightarrow Z$ and each morphism $f:P\rightarrow Z$, there exists a morphism $g:P\rightarrow Y$ such that $f=p\circ g$. The full subcategory of projective objects is denoted by $\mathcal{P}(\mathcal{C}, \CE)$, short only $\CP(\mathcal{C})$ if no confusion may arise. We say an exact category $(\mathcal{C},\mathcal{E})$ {\em has enough projective objects} if for each object $X\in\mathcal{C}$ there is an admissible epimorphism $p: P\rightarrow X$ with $P\in\mathcal{P}(\mathcal{C}, \CE)$. Dually, we can define injective objects. The full subcategory of injectives is denoted by $\mathcal{I}(\mathcal{C}, \CE)$, short only $\CI(\mathcal{C})$. An exact category is {\it Frobenius} provided that it has enough projectives and injectives, and moreover, the classes of projectives and injectives coincide.\\
 Let $(\mathcal{C}, \CE)$ be an exact category. For any $X, Z$ in $\mathcal{C}$, the extension group $\Ext^1_{\mathcal{C}}(Z,X)$ denotes the set of the Yoneda equivalences of short exact sequences $0\rightarrow X\rightarrow Y\rightarrow Z\rightarrow 0$ in $\mathcal{E}$.  We denote by $\underline{\mathcal{C}}$ the stable category of $\mathcal{C}$ which is defined as follows: the objects of $\underline{\mathcal{C}}$ are the same as $\mathcal{C}$ and the morphism space is defined by
$$\Hom_{\underline{\mathcal{C}}}(X, Y):= \Hom_{\mathcal{C}}(X, Y)/ \CP_{\CE}(X, Y)$$
for $X, Y$ in $\mathcal{C}$. We denote by $\CP_{\CE}(X, Y)$ the subgroup of $\Hom_{\mathcal{C}}(X, Y)$ which factor through an object of $\CP(\mathcal{C}, \CE)$. The costable category $\overline{\mathcal{C}}$ of $\mathcal{C}$ is defined dually in terms of the injective objects $\mathcal{I}(\mathcal{C}, \CE)$.\\
The bounded derived category of an exact category may  be defined in the same way as the usual one whenever the exact category by the Thomason's result to be  {\it weakly split
idempotent}. We recall that $\mathcal{C}$ is called weakly split idempotent provided  every retraction in $\mathcal{C}$ is an admissible epimonomorphism. Let $\mathbb{K}^{\rm{b}}(\mathcal{C})$ be the bounded  homotopy category of chain complexes of objects of $\mathcal{C}$. Let $\rm{Ac}^{\rm{b}}(\mathcal{C})$ be the full subcategory of $\mathbb{K}^{\rm{b}}(\mathcal{C})$
consisting of acyclic complexes. A complex $A=(A^n, d^n)$ over the exact category $\mathcal{C}$ is said to be {\it acyclic} if each differential $d^n$, $n \in  \mathbb{Z},$ factors as $A^n \st{p^n}\rt Z^{n+1}\st{q^{n+1}}\rt A^{n+1}$ such that $p^n$ and $q^{n+1}$ are respectively admissible epimorphism and admissible monomorphism, and also sequence $0 \rt Z^n \st{q^n}\rt A^n \st{p^n}\rt Z^{n+1}\rt 0$ lies in $\CE.$ The assumption of being weakly split idempotent implies the $\rm{Ac}^{\rm{b}}(\mathcal{C})$ is closed under direct summands. Then bounded derived catgery $\D^{\rm{b}}(\mathcal{C})$ is defined to be the {\it Verdier quotient}
$$\D^b(\mathcal{C})=\frac{\mathbb{K}^{\rm{b}}(\mathcal{C})}{\rm{Ac}^{\rm{b}}(\mathcal{C})}$$
as described, e.g. in Neeman \cite[Chapter 2]{N}. Inspired by the singularity category for $\mmod R$ over a Noetherian ring, which was studied by Buchweitz under the name of ``stable derived category'', we will define the singularity category for exact categories. The singularity category of the exact category $(\mathcal{C}, \CE)$ is by definition the Verdier quotient
$$\D_{\rm{sg}}(\mathcal{C})=\frac{\D^{\rm{b}}(\mathcal{C})}{\K^{\rm{b}}(\CP)}$$
where $\K^{\rm{b}}(\CP)$ is the homotopy category of bonded complexes of projective objects in $\mathcal{C}$.
\subsection{Morphism categories}
Assume that $\mathcal{A}$ is an abelian category. The {\em morphism category} of $\mathcal{A}$ is an abelian  category $\textup{Mor}(\mathcal{A})$ defined by the following data.
The objects of $\textup{Mor}(\mathcal{A})$ are all  the morphisms $f:X\rightarrow Y$ in $\mathcal{A}$. The morphisms from object $f:X\rightarrow Y$ to $f':X'\rightarrow Y'$ are pairs $(a,b)$, where $a:X\rightarrow X'$ and $b:Y\rightarrow Y'$ such that $b\circ f=f'\circ a$. The composition of morphisms is component-wise. We denote by $\textup{Mono}(\mathcal{A})$ the full subcategory of $\textup{Mor}(\mathcal{A})$ consisting of all monomorphisms in $\mathcal{A}$, which is called the {\em monomorphism category of $\mathcal{A}$}. Dually, we define the {\em epimorphism category} $\textup{Epi}(\mathcal{A})$ of $\mathcal{A}$.   $\textup{Mono}(\mathcal{A})$, resp. $\textup{Epi}(\mathcal{A})$,   is an additive category of $\textup{Mor}(\mathcal{A})$ which is closed under extensions, thus it becomes naturally an exact category. 
 In case that $\CA=\mmod \La$,  we know by a  general fact the category $\rm{H}(\mmod \La)$, or simply $\rm{H}(\La)$,  is equivalent to the category of finitely generated right modules over the path algebra $\La\mathbb{A}_2 \simeq T_2(\La)$, where  $T_2(\La)= \tiny {\left[\begin{array}{ll} \La & \La \\ 0 & \La \end{array} \right]}$, upper triangular $2\times 2$ matrix algebra over $\La.$ Because of this equivalence we identify the morphisms in $\mmod \La$ with the (finitely generated) right $T_2(\La)$-modules throughout  the paper. To consider a map $A \st{f}\rt B$ as an object in $\rm{H}(\La)$ we use the parentheses, and shown by $(A \st{f}\rt B)$. If no ambiguity may arise, without parentheses or just  with $f$ sometimes will be presented.
 We also will sometimes  write either
$$
(AB)_f \quad \text{or} \quad \begin{array}{c}
A \\
\mathrel{\mathop{\downarrow}^f}
\\
B
\end{array} $$
Especially, the first notation is used in diagrams. A morphism $(\phi_1, \phi_2): (A\st{f}\rt B)\rt (C\st{g}\rt D)$ in $\rm{H}(\La)$ is also visualized as 
{\footnotesize  \[\xymatrix@R-2pc {  &  A\ar[dd]^{f}  & C\ar[dd]~  \\   &  _{\ \ \ \ \   }\ar[r]^{\phi_1}_{\phi_2}  \ar[r]&_{\ \  }~ {g}   \\ & B & D  }\]}

Let $\CX$ be a subcategory of $\mmod \La$. Denote by $\CS_{\CX}(\La)$, resp. $\CF_{\CX}(\La)$, the full subcategory of $\rm{Mono}(\mmod \La)$, resp. $\rm{Epi}(\mmod \La)$, consisting of all monomorphisms, resp. epimorphisms,  $A \st{f}\rt B$ such that $A, B$ and $\rm{Cok} \ f$, resp. $\Ker \ f$,  belong to $\CX$. If $\CX=\mmod \La$, then $\CS_{\mmod \La}(\La)$, resp. $\CF_{\mmod \La}(\La)$, is nothing but $\rm{Mono}(\mmod \La)$, resp. $\rm{Epi}(\mmod \La)$. In this case, we suppress ``$\mmod \La$'' and only write $\CS(\La)$ and $\CF(\La)$, the similar notations used in \cite{RS2}. 

\subsection{An equivalence}\label{An equivalence}
Following \cite{MT}, a subcategory $\CX$ of $\mmod \La$ is called {\it quasi-resolving} if it satisfies the following conditions:
\begin{itemize}
	\item [$(1)$] $\CX$ contains $\rm{prj}\mbox{-}\La$;
	\item[$(2)$] $\CX$ is closed under finite direct sums, that is, for a finite number of objects $X_1, \cdots X_n$ in $\CX$ the direct sum $X_1\oplus\cdots\oplus X_n$ is in $\CX$;
	\item [$(3)$] $\CX$ is closed under kernels of epimorphisms, that is, for each short exact sequence $0 \rt  A \rt B \rt C \rt 0$, if $B$ and $C$ are in $\CX$, then so is $A$.
\end{itemize}
Moreover, a quasi-resolving subcategory $\CX$ is called {\it resolving} if it is closed under direct summand and extensions.

Let $\CX$ be a quasi-resolving subcategory of $\mmod \La$.  In \cite[Construction 3.1]{H2} the functor $\Psi_{\CX}:\CS_{\CX}(\La)\rt \mmod \underline{\CX}$ is defined as follows. Take a monomorphism $(A \st{f}\rt B)$ in $\CS_{\CX}(\La)$, let $\Psi_{\CX}(A\st{f}\rt B)$ be a fixed functor in $\mmod \CX$ which is fitted in the following exact sequence in $\mmod \CX$,
$$0 \rt(-, A)\st{(-, f)}\rt (-, B)\rt (-, \rm{Cok} \ f)\rt F \rt 0.$$
Since $F$ is obtained by a short sequence in $\mmod \La$ then it belongs to $\mmod \underline{\CX}$. It is proved that $\Psi_{\CX}$ is full, dense and objective. Hence it induces the following equivalence
$$\CS_{\CX}(\La)/\CU\simeq \mmod \underline{\CX},$$
where $\CU$ is the idempotent ideal generated by the objects in the form $X\st{1}\rt X$ or $0\rt X$ where $X$ runs through the object in $\CX$, for more details see \cite[Theorem 3.2]{H2}. In fact, the ideal $\CU$ is formed by the morphisms which factors through an object as a finite  direct sum of copies of those forms. We recall from \cite[Appendix]{RZ} that an additive functor $F: \CA\rt \CB$ between additive categories $\CA$ and $\CB$ is  said to be objective provided any morphism $f:A\rt A'$ in $\CA$ with $F(f)=0$ factors through a kernel object of $F$, i.e., $f=g\circ h$, for some $h:A \rt C$ and $g:C\rt B$ with $F(C)=0.$ Also, for an ideal $\CI$ of an additive category $\mathcal{C}$, we denote by $\mathcal{C}/\CI$ the corresponding factor category. It has the same objects as $\mathcal{C}$ and $\Hom_{\mathcal{C}/\CI}(X, Y)=\Hom_{\mathcal{C}}(X, Y)/\CI(X,  Y)$ for any  pair $X, Y$ of objects in $\mathcal{C}$.
\section{Three types of exact categories}\label{Section3}
In this section we will study three different kinds of exact structures on $\CS_{\CX}(\La)$ for a resolving and  contravariantly finite subcategory $\CX$ of $\mmod \La$. Then we will  characterize  the indecomposable projective and injective objects for each type of the exact categories, and moreover see  that these exact categories whether or not  have  enough projectives and injectives. Note that since we can consider $\CS_{\CX}(\La)$ as a subcategory of $\mmod T_2(\La)$ being closed under direct summands, hence we can conclude that $\CS_{\CX}(\La)$ has  Krull-Schmidt property.

From now on, let $\CX \subseteq \mmod \La$ be a resolving and contravariantly finite subcategory throughout the paper, except otherwise stated. However, in important parts of the paper to put the emphasis we recall the assumption of being resolving and contravariantly finite. Although, in some results we do not need assumption of being contravariantly finite. We need this assumption, in particular, for when we need $\mmod \CX$ to be abelian category. Since $\CX$ is closed under extensions then it inherits an exact structure from  $\mmod \La$ by the class of all short exact sequences in $\mmod \La$ with the terms in $\CX.$ In the rest, we consider $\CX$ as an exact category with this inherited exact structure. Denote by $\CP(\CX)$, resp. $\CI(\CX)$, the subcategory of $\CX$ consisting of the projective, resp. injective, objects in the canonical exact category $\CX$. It can be easily seen that $\CP(\CX)=\rm{prj}\mbox{-}\La.$ Hence the stable category $\underline{\CX}$ of the exact category $\CX$ is exactly the full subcategory of $\underline{\rm{mod}}\mbox{-}\La$ consisting of modules in $\CX.$ But, $\CI(\CX)$ may be different from $\rm{inj}\mbox{-}\La$, the subcategory of injective modules in $\mmod \La.$

\subsection{Case 1} Since $\CX$ is closed under extensions, then $\CS_{\CX}(\La)$  is closed under extensions in the morphism category $\rm{H}(\La)$. Hence $\CS_{\CX}(\La)$ becomes canonically an exact category by the class of all short exact sequences in $\rm{H}(\La)$ with terms in $\CS_{\CX}(\La)$.  Let us  denote by $\CE_{\CX}$ such short exact sequences.  When we call $\CS_{\CX}(\La)$ as an exact category in the paper we mean with the exact structure given by $\CE_{\CX}$.
\begin{proposition}\label{projective objects in monomorphism}
We have the following assertions.
	
	\begin{itemize}
		\item [$(1)$]
		A monomorphism $A\st{f}\rt B$ is an indecomposable projective object  in the exact  category $\CS_{\CX}(\La)$ if and only if it is isomorphic to either $P\st{1}\rt P$ or $0 \rt P$ with indecomposable projective  $P$ in $\mmod \La.$
		\item [$(2)$] Assume $\mathcal{X}$ has enough injectives. A monomorphism $A\st{f}\rt B$ is an indecomposable injective object in the exact  category $\CS_{\CX}(\La)$ if and only if it is isomorphic to either $I\st{1}\rt I$ or $0 \rt I$ with indecomposable injective  module $I$ in the exact category $\CX$.
	\end{itemize}
Moreover, the exact category $\CS_{\CX}(\La)$ has enough projectives, and if $\CX$ is  the same as $(2)$, then so $\CS_{\CX}(\La)$ has enough  injectives. 	
\end{proposition}
\begin{proof}
	The first part follows from \cite[Lemma 5.2]{H2}, although, in there,  the assumption of $\CX$ to be contravariantly finite is stated in the lemma but  is not really used. The latter part can be proved by the proof of the lemma where for any object in $\CS_{\La}(\CX)$ an admissible epimorphism is constructed. Dually we can observe for having enough injectives.
\end{proof}
The above result says the property of being Frobenius can be interchanged between $\CX$ and $\CS_{\CX}(\La)$.
\begin{corollary}
	Let $\CX$ be a Frobenius exact category. Then, the exact category  $\CX$ is Frobenius if and only if so is the exact category $\CS_{\CX}(\La)$.
\end{corollary}
\subsection{Case 2}
 Let $\mathcal{E}^{\rm{cw}}_{\CX} \subseteq \CE_{\CX}$ be the class of all exact sequences

{\footnotesize  \[ \xymatrix@R-2pc {  &  ~ X_1\ar[dd]^{f}~   & Z_1\ar[dd]^{h}~  & Y_1\ar[dd]^{g} \\ 0 \ar[r] &  _{ \ \ \ \ } \ar[r]^{\phi_1}_{\phi_2}  &_{\ \ \ \ \ } \ar[r]^{\psi_1}_{\psi_2} _{\ \ \ \ \ }&  _{\ \ \ \ \ }\ar[r] & 0 \\ & X_2 & Z_2 & Y_2 }\]}
such that the sequences $0 \rt X_i\st{\phi_i}  \rt Z_i  \st{\psi_i} \rt Y_i \rt 0 $, $i=1, 2,$ in $\CX$ are split. Indeed, it is not difficult to see that any short exact sequence in $\mathcal{E}^{\rm{cw}}_{\CX}$  is isomorphic to a short exact sequence in the following form
{\footnotesize  \[ \xymatrix@R-2pc {  &  ~ X_1\ar[dd]^{f}~   & X_1\oplus Y_1\ar[dd]^{p}~  & Y_1\ar[dd]^{g} \\ 0 \ar[r] &  _{ \ \ \ \ } \ar[r]^{[1~~0]^t }_{[1~~0]^t}  &_{\ \ \ \ \ } \ar[r]^{[0~~1]}_{[0~~1]} _{\ \ \ \ \ }&  _{\ \ \ \ \ }\ar[r] & 0 \\ & X_2 & X_2\oplus Y_2 & Y_2 }\]}
where
${p}= \tiny {\left[\begin{array}{ll} f & q \\ 0 & g \end{array} \right]}$ with $q:Y_1 \rt X_2$. To make  our argument easier  without loss of generality we usually work with the short exact sequences in the above form when dealing with the short exact sequences in $\mathcal{E}^{\rm{cw}}_{\CX}$.
\begin{lemma}\label{Lemma 3.2}
	The class $\mathcal{E}^{\rm{cw}}_{\CX}$, defined in the above, is closed under pull-backs and push-outs taken in $\rm{H}(\La).$
\end{lemma}
\begin{proof} We only prove to be closed under push-outs.
 Assume that
	$0 \rt (X_1 \st{f}\rt X_2)\rt (Z_1 \st{h}\rt Z_2) \rt (Y_1 \st{g}\rt Y_2)\rt 0$ is in $\mathcal{E}^{\rm{cw}}_{\CX}$ and $\varphi=(\varphi_1, \varphi_2):(X_1 \st{f}\rt X_2) \rt (X'_1 \st{f}\rt X'_2)$ a morphism in $\CS_{\CX}(\La)$. By taking push-out  of the short exact sequence along $\varphi=(\varphi_1, \varphi_2)$ in $\rm{H}(\La)$, which exists definitely, then there is the following commutative diagram in $\rm{H}(\La)$
	$$\xymatrix
	{0\ar[r] & (X_1X_2)_f\ar[r]\ar[d]^{\varphi} & (Z_1Z_2)_h\ar[r]\ar[d]^{\varphi'} & (Y_1Y_2)_g\ar[r] \ar@{=}[d] & 0\\
		0\ar[r] & (X'_1X'_2)_{f'}\ar[r] & (Z'_1Z'_2)_{h'} \ar[r] & (Y_1Y_2)_g\ar[r] & 0.
	}$$
By the above diagram we have the following diagrams
$$\xymatrix
{0\ar[r] & X_i\ar[r]\ar[d]^{\varphi_i} & Z_i\ar[r]\ar[d] & Y_i\ar[r] \ar@{=}[d] & 0 &(\dagger_i)&\\
	0\ar[r] & X'_i\ar[r] & Z'_i \ar[r] & Y_i\ar[r] & 0&&
}$$

in $\mmod \La$ for $i=1, 2.$ On the other hand, by our assumption, the top row in the  digram $(\dagger_i)$ is split, and so is  the low row. Hence each row in the short exact sequence $0 \rt (X'_1 \st{f'}\rt X'_2)\rt (Z'_1 \st{h'}\rt Z'_2) \rt (Y_1 \st{g}\rt Y_2)\rt 0$ is split, consequently it lies in $\mathcal{E}^{\rm{cw}}_{\CX}$. So we are done.
\end{proof}
\begin{corollary}\label{corollary 3.3}
	The class $\mathcal{E}^{\rm{cw}}_{\CX}$ gives an exact structure on the additive category  $\CS_{\CX}(\La)$.
\end{corollary}
\begin{proof}
	The class $\mathcal{E}^{\rm{cw}}_{\CX}$ is clearly closed under isomorphisms and holds the axiom $\rm{E}_0$ and its dual. For $\rm{E}_1$, we use this easy fact, that is, the composition of splittable epimorphisms is again split. We can also conclude dually for $(\rm{E}_1)^{\rm{op}}$. The statement  follows from Lemma \ref{Lemma 3.2} in conjugation with Lemma \ref{Lemma2.1}.
\end{proof}
$\CS^{\rm{cw}}_{\CX}(\La)$ stand for the exact category $(\CS_{\CX}(\La),
\mathcal{E}^{\rm{cw}}_{\CX})$.

\begin{proposition}\label{Proposition 3.4}
	Let $(X_1 \st{f} \rt X_2)$ be an object in $\CS^{\rm{cw}}_{\CX}(\La)$. Then $f$ is an indecomposable projective   object in the exact category $\CS^{\rm{cw}}_{\CX}(\La)$ if and only if it is isomorphic to either of the following objects
	\begin{itemize}
		\item [$(1)$] $X \st{1}\rt X $ for some  indecomposable object $X$ in $\CX;$
		\item[$(2)$] $0 \rt X$ for some  indecomposable  object $X$ in $\CX$.
	\end{itemize}
Moreover, $\CS^{\rm{cw}}_{\CX}(\La)$ has enough projectives.
\end{proposition}
\begin{proof}For each $A \in \mmod \La,$  in an easy way,  we observe  that the endomorphism algebra of $(A \st{1}\rt A)$, resp. $(0 \rt A)$,  is isomorphic to $\End_{\La}(A)$. Thus, if $A$ is an indecomposable module, then the  endomorphism algebra of $(A \st{1}\rt A)$, resp. $(0 \rt A)$,  is local. Therefore, for an  indecomposable object $X \in \CX$,  the object $(X \st{1}\rt X)$, resp. $(0 \rt X)$,  is indecomposable in   $\CS^{\rm{cw}}_{\CX}(\La)$. Let $X$ be an indecomposable module in $\CX$, and  take an arbitrary  short exact sequence in $\mathcal{E}^{\rm{cw}}_{\CX}$ of  the following form
	{\footnotesize  \[ \xymatrix@R-2pc {  &  ~ X_1\ar[dd]^{f}~   & X_1\oplus Y_1\ar[dd]^{p}~  & Y_1\ar[dd]^{g} \\ 0 \ar[r] &  _{ \ \ \ \ } \ar[r]^{[1~~0]^t }_{[1~~0]^t}  &_{\ \ \ \ \ } \ar[r]^{[0~~1]}_{[0~~1]} _{\ \ \ \ \ }&  _{\ \ \ \ \ }\ar[r] & 0. \\ & X_2 & X_2\oplus Y_2 & Y_2 }\]}
	where ${p}= \tiny {\left[\begin{array}{ll} f & q \\ 0 & g \end{array} \right]}$ with $q:Y_1 \rt X_2$. To prove the ``if'' part, based on the above observation,  we only need to show that any morphisms
	
	\begin{align*}
	{\footnotesize  \xymatrix@R-2pc {  &  X\ar[dd]^{1}  & Y_1\ar[dd]~  \\   &  _{\ \ \ \ \   }\ar[r]^{d_1}_{d_2}  \ar[r]&_{\ \  }~ {g}   \\ & X & Y_2  }} &   {\footnotesize  \xymatrix@R-2pc {  &  0\ar[dd]  & Y_1\ar[dd]~  \\   &  _{\ \ \ \ \   }\ar[r]^{0}_{d}  \ar[r]&_{\ \  }~ {g}   \\ & X & Y_2  }}  	
	\end{align*}
	
	factor through the morphism
	{\footnotesize \[ \xymatrix@R-2pc {  &  X_1\oplus Y_1\ar[dd]^{p}  & Y_1\ar[dd]~  \\   &  _{\ \ \ \ \ \  \ \ \  }\ar[r]^{[0~~1]}_{[0~~1]}  \ar[r]&_{\ \  }~ {g}   \\ & X_2\oplus Y_2 & Y_2  }\]}
	To this end, it is enough to take the following morphisms for each of them, respectively, to factor the given morphisms
	
	\begin{align*}
	{\footnotesize  \xymatrix@R-2pc {  &  X\ar[dd]^{1}  & X_1 \oplus Y_1\ar[dd]~  \\   &  _{\ \ \ \ \  }\ar[r]^{[0~~d_1]^t}_{[qd_1~~~d_2]^t}  \ar[r]&_{\  }~~~{p}   \\ &  X & X_2\oplus Y_2  }} & {\footnotesize  \xymatrix@R-2pc {  &  0\ar[dd]  & X_1 \oplus Y_1\ar[dd]~  \\   &  _{\ \ \ \ \  }\ar[r]^{0}_{[0~~d]^t}  \ar[r]&_{\  }~~~{p}   \\ &  X & X_2\oplus Y_2  }}
	\end{align*}
	
	To prove the ``only if'', consider the following short exact sequence in $\mathcal{E}^{\rm{cw}}_{\CX}$, which exists   for any object $(X_1 \st{f}\rt X_2)$
	{\footnotesize  \[ \xymatrix@R-2pc {&& &  ~ 0\ar[dd]^{f}~   & X_1\ar[dd]^{[1~~0]^t}~  & X_1\ar[dd]^{f} &&(\dagger)&\\ &&0 \ar[r] &  _{ \ \ \ \   \ \ \                  } \ar[r]^>>>>>>>{0 }_>>>>>>{[1~~f]^t}  &_{\ \ \ \ \ \ \ \ \ \ \  \ \ } \ar[r]^{1}_{[f~~-1]} _{\ \ \ \ \ }&  _{\ \ \ \ \ }\ar[r] & 0.& && \\&& & X_1  & X_1\oplus X_2 & X_2&&& }\]}
	 Now by using this short exact sequence we see that any projective object in the exact category  $\CS^{\rm{cw}}_{\CX}(\La)$ is isomorphic to one of the indecomposable objects  appeared in either $(1)$ or $(2)$. The short exact sequence $(\dagger)$ also   follows $\CS^{\rm{cw}}_{\CX}(\La)$ has enough projectives. Now the proof is complete.
\end{proof}
\begin{proposition}\label{Proposition 3.5}
	Assume that
	the exact category $\CX$ has enough injectives.
	Let $(X_1 \st{f} \rt X_2)$ be an object in $\CS^{\rm{cw}}_{\CX}(\La)$. Then $f$ is an indecomposable
	injective  object in the exact category $\CS^{\rm{cw}}_{\CX}(\La)$ if and only if it is isomorphic to either of the following objects
	\begin{itemize}
		\item [$(1)$] $X \st{1}\rt X $ for some indecomposable object $X$ in $\CX;$
		\item [$(2)$] 		
		$X \st{l}\rt I$ for some indecomposable object $X$ in $\CX,$ where $l$ is a  left minimal  monomorphism, i.e., any endomorphism $g$ with $g\circ l=l$ is an automorphism, 
		 and $I $ an injective object in the exact category $\CX$;
		\item [$(3)$] 	$0 \rt I$ for some indecomposable injective object $I$ in the exact category $\CX$.
	\end{itemize}
\end{proposition}
Moreover, $\CS^{\rm{cw}}_{\CX}(\La)$ has enough injectives.
\begin{proof} First we show the ``if'' part.
	Let $X$ be an indecomposable object in $\CX.$ Let also $(X \st{l}\rt I)$ be  the same as in $(2)$ of the statement. The object $(X \st{l}\rt I)$ is indecomposable in $\CS^{\rm{cw}}_{\CX}(\La)$; otherwise, since $X$ is indecomposable then there was a decomposition as $(X \st{l}\rt I)=(X \st{l'}\rt I')\oplus (0 \rt J)$ with $J\neq 0$. But this is a contraction by $l$ being left minimal.  Consider  an arbitrary  short exact sequence in the following form
	{\footnotesize  \[ \xymatrix@R-2pc {  &  ~ X_1\ar[dd]^{f}~   & X_1\oplus Y_1\ar[dd]^{p}~  & Y_1\ar[dd]^{g} \\ 0 \ar[r] &  _{ \ \ \ \ } \ar[r]^{[1~~0]^t }_{[1~~0]^t}  &_{\ \ \ \ \ } \ar[r]^{[0~~1]}_{[0~~1]} _{\ \ \ \ \ }&  _{\ \ \ \ \ }\ar[r] & 0 \\ & X_2 & X_2\oplus Y_2 & Y_2 }\]}
	in $\mathcal{E}^{\rm{cw}}_{\CX}$, where  ${p}= \tiny {\left[\begin{array}{ll} f & q \\ 0 & g \end{array} \right]}$ with $q:Y_1 \rt X_2$. To show that the  object $(X \st{l}\rt I)$ to be injective in $\CS^{\rm{cw}}_{\CX}(\La)$, it suffices to prove that any morphism {\footnotesize  \[\xymatrix@R-2pc {  &  X_1\ar[dd]^{f}  & X\ar[dd]~  \\   &  _{\ \ \ \ \   }\ar[r]^{d_1}_{d_2}  \ar[r]&_{\ \  }~ {l}   \\ & X_2 & I  }\]}
	factors through the morphism
	{\footnotesize \[ \xymatrix@R-2pc {  &  X_1\ar[dd]^{f}  & X_1\oplus Y_1\ar[dd]~  \\   &  _{\ \ \ \ \ \  \ \ \  }\ar[r]^{[1~~0]^t}_{[1~~0]^t}  \ar[r]&_{\ \  }~ {p}   \\ & X_2& X_2\oplus Y_2  }\]}
	For this, the following morphism works	
	{\footnotesize \[ \xymatrix@R-2pc {  &  X_1\oplus Y_1 \ar[dd]^{p}  & X\ar[dd]~  \\   &  _{\ \ \ \ \ \  \ \ \  }\ar[r]^{[d_1~~0]}_{[d_2~~-r]}  \ar[r]&_{\ \  }~ {l}   \\ & X_2\oplus Y_2 &I  }\]}
	where $r:Y_2 \rt I$ holds in the equality $rg=d_2q$, and it exists because of $I$  being injective in the exact category $\CX$. Let $0 \rt I$ be  the same as in $(3)$. For any arbitrary morphism
	{\footnotesize \[ \xymatrix@R-2pc {  &  X_1 \ar[dd]^{f}  & 0\ar[dd]~  \\   &  _{\ \ \ \ \ \  \ \ \  }\ar[r]^{0}_{t}  \ar[r]&_{\ \  }~ {}   \\ & X_2&I  }\]}
	We may  consider the following morphism 	{\footnotesize \[ \xymatrix@R-2pc {  &  X_1\oplus Y_1 \ar[dd]^{p}  & 0\ar[dd]~  \\   &  _{\ \ \ \ \ \  \ \ \  }\ar[r]^{0}_{[t~~-s]}  \ar[r]&_{\ \  }~ {}   \\ & X_2\oplus Y_2 &I  }\]}
	where $s:Y_2 \rt I$ is a morphism such that $sg=tq$, which exists again  because of $I$ being  injective. Finally,  in a dual way of the previous proposition one can show that an  object $(X \st{1}\rt X)$ the same as in $(1)$ is injective in the exact category $\CS^{\rm{cw}}_{\CX}(\La)$. Therefore, the proof of this part is completed.
	
	 For ``only if'', for any given object $(X_1 \st{f}\rt X_2)$,  we  first construct a  short exact sequence in $\CE^{\rm{scw}}_{\CX}$ starting at it. Since the exact category $\CX$ has enough injective then there is a short exact sequence $0 \rt X_2\st{m}\rt I\rt Y \rt 0$ with terms in $\CX$ and   $I$ an injective object in the exact category . Consider the following push-out diagram
	\[\xymatrix@C-0.5pc@R-.8pc{&0\ar[d]&0 \ar[d] &0 \ar[d] &&\\0\ar[r]& X_1
		\ar[r]^f \ar@{=}[d] & X_2 \ar[d]^m\ar[r] & \rm{Cok} \ f \ar[d]\ar[r]&
		0& \\ 0 \ar[r]  & X_1 \ar[r]  & I\ar[d]\ar[r] & D \ar[d] \ar[r]&0&\\  &  & Y\ar@{=}[r] \ar[d] &Y \ar[d]&&\\ && 0& 0& &\\}\]
	
	Since $\CX$ is closed under extensions, then $D$ belongs to $\CX$, and consequently  $(X_1 \st{mf} \rt I)$ lies in $\CS_{\CX}(\La)$. On the other hand, there is a decomposition $I=I_1\oplus I_2$ such that $p_1mf:X_1 \rt I_1$ is left minimal and $p_2mf:X_1 \rt I_2$ is zero, where $p_1$ and $p_2$ are the canonical projections. Then due to the decomposition we can identify $(X_1 \st{mf} \rt I)$ as the direct sum $(X_1\st{p_1mf}\lrt I_1)\oplus (0\rt I_2)$.  Now we are ready to introduce the promised short exact sequence in $\mathcal{E}^{\rm{cw}}_{\CX}$ as the following
	{\footnotesize  \[ \xymatrix@R-2pc {  &  ~ X_1\ar[dd]^{f}~   & X_1\oplus X_2\ar[dd]^{h}~  & X_2\ar[dd]^{m} \\ 0 \ar[r] &  _{ \ \ \ \ } \ar[r]^{[1~~f]^t }_{[m~~1]^t}  &_{\ \ \ \ \ } \ar[r]^{[-f~~1]}_{[-1~~m]} _{\ \ \ \ \ }&  _{\ \ \ \ \ }\ar[r] & 0 \\ & X_2 & I\oplus X_2 & I }\]}
	
	where  ${h}= \tiny {\left[\begin{array}{ll} mf & 0 \\ 0 & 1 \end{array} \right]}$. Trivially, the middle  object $h$  is isomorphic to the direct sum $(X_1 \st{p_1mf}\lrt I_1) \oplus (0\rt I_2)\oplus (X_2 \st{1}\rt X_2)$. The above short exact sequence gives  the proof since if the object  $f$ is injective, then it must be a direct summand of $h$, so to be isomorphic to  one of the forms $(1), (2)$ and $(3)$ in the statement. Moreover,  the last part of the statement follows from the short exact sequence.
\end{proof}
\begin{definition}\label{Definition 3.7}	
	For a given exact category $(\mathcal{C}, \CE)$ with enough projetives, resp. injectives,   we introduce the following   dimensions for any $C\in\mathcal{C}.$
	\begin{itemize}
		\item[(a)] A complex over $\mathcal{C}$
		$$0 \rt P_n \st{d_n}\rt  P_{n-1} \st{d_{n-1}}\lrt  \cdots \rightarrow P_1\st{d_1} \rightarrow P_0 \st{d_0}\rt C \rt 0,$$
		is called an (finite) {\it $\CE$-projective resolution} of $C$ in the exact category $\mathcal{C}$ if it is acyclic over $\mathcal{C}$ and $P_i$, $0 \leqslant i \leqslant n$, belongs to $\CP(\mathcal{C}).$ The value of $n$ is called the length of the $\CE$-projective resolution.
		\item[(b)] The {\it $\CE$-projective dimension} $\rm{pd_{\CE}}C$ of $C$ is the minimal non-negative integer $n$ such that there is an $\CE$-projective resolution of $\mathcal{C}$ with length $n$. If $C$ does not admit a finite $\CE$-projective resolution, then by convention the $\CE$-projective dimension of $C$ is said to be infinite.
	\end{itemize}
\end{definition}
 One can define dually the {\it $\CE$-injective  dimension} $\rm{id}_{\CE} C$ of an object $C$ for those exact categories with enough injectives.

Motivated by  hereditary algebras, we can define the following definition for the exact categories.
\begin{definition}
	Let $(\mathcal{C}, \CE )$ be an exact category with enough projectives and injectives. The exact category $\mathcal{C}$ is called hereditary if for any object $C$, $\rm{pd_{\CE}}C$ and  $\rm{id}_{\CE} C \leqslant 1.$
\end{definition}
\begin{proposition}
Assume the exact category $\CX$ has enough injectives. 	The  $\CS^{\rm{cw}}_{\CX}(\La)$  is  a hereditary exact category.
\end{proposition}
\begin{proof}
It follows from Proposition \ref{Proposition 3.4} and \ref{Proposition 3.5} that $\CS^{\rm{cw}}_{\CX}(\La)$ has enough projectives and injectives. In addition,
 the proof of the propositions for any object $(X_1 \st{f}\rt Y_1)$ are proved by the existence of  the following short exact sequences in $\CS^{\rm{cw}}_{\CX}(\La)$
{\footnotesize  \[ \xymatrix@R-2pc {&& &  ~ 0\ar[dd]~   & X_1\ar[dd]^{[1~~0]^t}~  & X_1\ar[dd]^{f} &&\\ &&0 \ar[r] &  _{ \ \ \ \   \ \ \                  } \ar[r]^>>>>>>>{0 }_>>>>>>{[1~~f]^t}  &_{\ \ \ \ \ \ \ \ \ \ \  \ \ } \ar[r]^{1}_{[f~~-1]} _{\ \ \ \ \ }&  _{\ \ \ \ \ }\ar[r] & 0& \\&& & X_1  & X_1\oplus X_2 & X_2& }\]}
and
	{\footnotesize  \[ \xymatrix@R-2pc {  &  ~ X_1\ar[dd]^{f}~   & X_1\oplus X_2\ar[dd]^{h}~  & X_2\ar[dd]^{m} \\ 0 \ar[r] &  _{ \ \ \ \ } \ar[r]^{[1~~f]^t }_{[m~~1]^t}  &_{\ \ \ \ \ } \ar[r]^{[-f~~1]}_{[-1~~m]} _{\ \ \ \ \ }&  _{\ \ \ \ \ }\ar[r] & 0 \\ & X_2 & I\oplus X_2 & I }\]}
where  ${h}= \tiny {\left[\begin{array}{ll} mf & 0 \\ 0 & 1 \end{array} \right]}$. The above short exact sequences implies that $\rm{pd}_{\mathcal{E}^{\rm{cw}}_{\CX}}(X_1 \st{f}\rt X_2)$ and $\rm{id}_{\mathcal{E}^{\rm{cw}}_{\CX}}(X_1 \st{f}\rt X_2)$ are at most one. Hence
the exact category
 $\CS^{\rm{cw}}_{\CX}(\La)$ is   hereditary.	
\end{proof}

\subsection{ Case 3}
Let
{\footnotesize  \[ \xymatrix@R-2pc {  &  ~ X_1\ar[dd]^{f}~   & Z_1\ar[dd]^{h}~  & Y_1\ar[dd]^{g} \\ \epsilon: \  \ \  \  0 \ar[r] &  _{ \ \ \ \ } \ar[r]^{\phi_1}_{\phi_2}  &_{\ \ \ \ \ } \ar[r]^{\psi_1}_{\psi_2} _{\ \ \ \ \ }&  _{\ \ \ \ \ }\ar[r] & 0 \\ & X_2 & Z_2 & Y_2 }\]}
be a short exact sequence in $\CE_{\CX}$. Expanding the above sequence in $\mmod \La$, we get the following commutative diagram
	$$\xymatrix{& 0 \ar[d] & 0 \ar[d] & 0 \ar[d]& (*)&\\
	0 \ar[r] & X_1 \ar[d]^f \ar[r]^{\phi_1} & Z_1 \ar[d]^h
	\ar[r]^{\psi_1} & Y_1 \ar[d]^g \ar[r] & 0\\
	0 \ar[r] & X_2\ar[d] \ar[r]^{\phi_2} & Z_2
	\ar[d] \ar[r]^{\psi_2} & Y_2 \ar[d] \ar[r] & 0\\
	0 \ar[r] & \rm{Cok}\ f \ar[d] \ar[r]^{\mu_1} & \rm{Cok} \ h
	\ar[d] \ar[r]^{\mu_2} & \rm{Cok}\ g \ar[d] \ar[r] & 0\\
	& 0  & 0  & 0 & }
$$
in $\mmod \La.$ Denote by $\CE^{\rm{scw}}_{\CX}$, the class of all those short exact sequences $\epsilon$ in $\CE_{\CX}$ such that   all the rows in their induced diagram $(*)$ in $\mmod \La$  are split. Clearly, $\CE^{\rm{scw}}_{\CX} \subseteq \CE^{\rm{cw}}_{\CX}.$
\begin{lemma}\label{Lemma 3.9}
	The class $\mathcal{E}^{\rm{scw}}_{\CX}$, defined in the above, is closed under pull-backs and push-outs taken in $\rm{H}(\La).$
\end{lemma}
\begin{proof}
	The proof is the same as Lemma \ref{Lemma 3.2}.
\end{proof}
The same as Corollary \ref{corollary 3.3} we have:
\begin{corollary}
	The class $\mathcal{E}^{\rm{scw}}_{\CX}$ gives an exact structure on the additive category  $\CS_{\CX}(\La)$.
\end{corollary}
Let   $\CS^{\rm{scw}}_{\CX}(\La)$ denote the exact category $(\CS_{\CX}(\La), \mathcal{E}^{\rm{scw}}_{\CX})$. For the case $\CX=\mmod \La$, we usually write only $\CS^{\rm{scw}}(\La)$.
\begin{remark}\label{RemarkSToF}
	In analog with the three  types of exact structures defined in the above on the additive category $\CS_{\CX}(\La)$,  one can also define three  types exact structures $\CF_{\CX}(\La), \CF^{\rm{cw}}_{\CX}(\La), \CF^{\rm{s
			cw}}_{\CX}(\La)$  on the additive category $\CF_{\CX}(\La)$, respectively. The kernel and cokernel functors, denoted by $\CK\rm{er}$ and $\mathcal{C}\rm{ok}$ respectively,  induce  mutually inverse equivalences between additive categories $\CS_{\CX}(\La)$ and $\CF_{\CX}(\La)$. By the Snake  lemma, it can be seen that these two equivalences are exact functors, whenever $\CS_{\CX}(\La)$ and $\CF_{\CX}(\La)$ regarded as exact categories   induced by the abelian structure of the morphism category $\rm{H}(\La)$. Also, these two equivalences preserve the  exact structures in the exact categories $\CS_{\CX}^{\rm{scw}}(\La)$ and $ \CF^{\rm{s
			cw}}_{\CX}(\La)$. Hence $\CK\rm{er}$ and $\mathcal{C}\rm{ok}$ are mutually inverse equivalences of of the first and third types of the exact categories on $\CS_{\CX}(\La)$ and $\CF_{\CX}(\La)$.
\end{remark}

Let $M$ be a module in $\mmod \La$ and $P_M \st{\pi_M}\rt M$ the projective cover of $M$ in $\mmod \La.$ Set $\Omega_{\La}(M)$ the kernel of $\pi_M$ and $\Omega_{\La}(M)\hookrightarrow P_M$ the canonical inclusion.
\begin{proposition}\label{Proposition 3.4}
	Let $(X_1 \st{f} \rt X_2)$ be an object in $\CS_{\CX}(\La)$. Then $f$ is an indecomposable projective   object in the exact category $\CS^{\rm{scw}}_{\CX}(\La)$ if and only if it is isomorphic to either of the following objects
	\begin{itemize}
		\item [$(1)$] $X \st{1}\rt X $ for some  indecomposable object $X$ in $\CX;$
		\item[$(2)$] $0 \rt X$ for some  indecomposable  object $X$ in $\CX$;
		\item [$(3)$] $\Omega_{\La}(X)\hookrightarrow P_X$ for some indecomposable  object $X$ in $\CX$.
	\end{itemize}
	Moreover, $\CS^{\rm{scw}}_{\CX}(\La)$ has enough projectives.
\end{proposition}
\begin{proof} First note that an object in $\CS_{\CX}(\La)$ which is isomorphic to the one given in $(3)$ of the  statement is indecomposable. It follows by  the indecomposability of $X$ and $\pi_X$ being projective cover. The indecomposability of the other cases $(1)$ and $(2)$ have been discussed before. These two cases are also projective in the $\CS^{\rm{scw}}_{\CX}(\La)$ since $\mathcal{E}^{\rm{scw}}_{\CX} \subseteq \mathcal{E}^{\rm{cw}}_{\CX}$. To prove the case $(3)$ Remark \ref{RemarkSToF} comes to play as follows. In the dual way of Proposition \ref{Proposition 3.5}(2), one can prove $P \st{\pi_X}\rt X$ is projective in the exact category $\CF^{\rm{cw}}_{\CX}(\La)$. Let $\epsilon$ be an arbitrary short  exact sequence in $\mathcal{E}^{\rm{scw}}_{\CX}$ and $\phi$ an arbitrary morphism from $\Omega_{\La}(X)\hookrightarrow P_X$ to the end term of the $\epsilon.$ Applying the Cokernel functor on $\epsilon$ and $\phi$ gives the  exact sequence $\epsilon'$ in $\CF^{\rm{cw}}_{\CX}(\La)$ and the morphism $\phi'$ from $P \st{\pi_X}\rt X$ to the end term of $\epsilon'$. As stated $P \st{\pi_X}\rt X$ is an projective object in the exact category $\CF^{\rm{cw}}_{\CX}(\La)$, so $\phi'$ factors through the epimorphism included in the $\epsilon'$. Then by applying the Kernel  functor on this factorization we get the desired factorization in $\CS^{\rm{scw}}_{\CX}(\La)$. This completes the proof of the ``if'' part. For the converse part, let $(X_1 \st{f} \rt X_2)$ be an indecomposable projective  object in $\CS^{\rm{scw}}_{\CX}(\La)$. By applying the functor $\Psi_{\CX}$ on $(X_1 \st{f} \rt X_2)$ we will get the functor $F$ in $\mmod \underline{\CX}$ as fitted  in the following exact sequence in $\mmod \CX$
	$$0 \rt (-, X_1)\rt (-, X_2) \rt (-, \rm{Cok} \ f) \rt F \rt 0.$$
	The above sequence induces the short exact sequence $0\rt K \rt (-, \underline{\rm{Cok}\ f}) \rt F \rt 0$ in $\mmod \underline{\CX}$. Since $K$ is in $\mmod \underline{\CX}$ then there is a projective resolution in $\mmod \CX$ as the following
	$$0 \rt (-, Y_1)\rt (-, Y_2) \rt (-,Y_3) \rt K \rt 0.$$

 By the horseshoe lemma  we have the following  commutative diagram in $\mmod \CX$ with exact rows and the first three columns (from the left side hand) splitting:
$$\xymatrix{& 0 \ar[d] & 0 \ar[d] & 0 \ar[d]& 0\ar[d]& &\\
	0 \ar[r] & (-, Y_1) \ar[d] \ar[r] & (-, Y_2) \ar[d]
	\ar[r] &(-, Y_3) \ar[d] \ar[r] & K\ar[r]\ar[d]&0\\
	0 \ar[r] & (-, X_1\oplus Y_1)  \ar[r]^{(-,h)}\ar[d] & (-, X_2\oplus Y_2)\ar[d]
	\ar[r] & (-, \rm{Cok} \ f \oplus Y_3) \ar[r]\ar[d] & (-, \underline{\rm{Cok} f} \ )\ar[r]\ar[d]& 0\\0\ar[r] &(-, X_1)\ar[r]^{(-, f)}\ar[d]&(-,  X_2)\ar[r]\ar[d]&(-, \rm{Cok} \ f)\ar[d] \ar[r]& F \ar[r]\ar[d]&0 &\\ &0&0&0&0&&
} 	 $$	
By the above diagram we have the middle row plays a projective resolution of $(-, \underline{\rm{Cok} \ f})$	in $\mmod \CX$. On the other hand,  we know minimal projective resolution of $(-, \underline{\rm{Cok} \ f})$	in $\mmod \CX$ is as the following
$$0 \rt (-, \Omega_{\La}(\rm{Cok} \ f)) \rt (-, P_{\rm{Cok} \ f})\rt (-, \rm{Cok} \ f)\rt(-, \underline{\rm{Cok} \ f})\rt  0.$$
Thus, the deleted projective resolution appeared in the middle row of the above diagram is the direct sum of the deleted minimal one and some contractible complexes. Thus, we can write the object $(X_1\oplus Y_1 \st{h}\rt X_2\oplus Y_2)$ in $\CS_{\CX}(\La)$ as $(\Omega_{\La}(\rm{Cok} \ f) \rt P_{\rm{Cok} \ f})\oplus (0 \rt M)\oplus (N \st{1}\rt N)$ for some $M$ and $N$ in $\CX$. This observation follows the middle term of the short exact sequence in $\CE^{\rm{scw}}_{\CX}$ ending at the object $f$, which is  induced by the above diagram, and then using the Yoneda lemma, is isomorphic to  a direct sum of the indecomposable objects in the three types as presented in the assertion.  But since $f$ is projective by our assumption then it must be a direct summand of the middle term, so we are done.	
\end{proof}
For a module $M$ in $\mmod \La$, $\Ext^1_{\La}(-, M)|_{\CX}$  denotes  the restriction of the functor $\Ext^1_{\La}(-, M)$ in $\mmod \mmod \La$ on the subcategory $\CX.$ 
\begin{lemma}\label{Lemma 3.19}
Assume that
the exact category $\CX$ has enough injectives. For any $X \in \CX$, $\Ext^1_{\La}(-, X)|_{\CX}$ is an injective object in $\mmod \underline{\CX}$. Also if $F$ is in $\mmod \underline{\CX}$, then there is a monomorphism $0 \rt F \rt \Ext^1_{\La}(-, X)|_{\CX}$ for some $X$ in $\CX.$
\end{lemma}
\begin{proof}
	
Let $X$ be in $\CX$. Since $\CX$ has enough injectives then there is a short exact sequence $0 \rt X \rt I \rt L \rt 0$ in $\CX$ with $I$ an injective object in the exact category $\CX$. This short exact sequence induces the following exact sequence in $\mmod \CX,$
$$0 \rt (-, X)\rt (-, I)\rt (-, L)\rt \Ext^1_{\La}(-, X)|_{\CX}\rt 0.$$
The above sequence says $\Ext^1_{\La}(-, X)|_{\CX}$ lies in $\mmod \CX$, and equivalently in $\mmod \underline{\CX}$. The proof of the other assertions is the same as corresponding results given in \cite[ section 4]{A}.	
\end{proof}
The next result follows by  the dual of
argument given in Proposition \ref{Proposition 3.4} and in view of  Lemma \ref{Lemma 3.19}.
\begin{proposition}\label{Proposition 3.15}
	Assume that
	the exact category $\CX$ has enough injectives.
	Let $(X_1 \st{f} \rt X_2)$ be an object in $\CS^{\rm{scw}}_{\CX}(\La)$. Then $f$ is an indecomposable
	injective  object in the exact category $\CS^{\rm{scw}}_{\CX}(\La)$ if and only if it is isomorphic to either of the following objects
	\begin{itemize}
		\item [$(1)$] $X \st{1}\rt X $ for some indecomposable object $X$ in $\CX;$
		\item [$(2)$] 		
		$X \st{l}\rt I$ for some indecomposable object $X$ in $\CX,$ where $l$ is a  left minimal  monomorphism and $I $ an injective object in the exact category $\CX$;
		\item [$(3)$] 	$0 \rt X$ for some indecomposable  object  in the exact category $\CX$.
	\end{itemize}
Moreover, $\CS^{\rm{scw}}_{\CX}(\La)$ has enough injectives.
\end{proposition}
The characterization of projective, or injective, objects in the exact categories $\CS_{\CX}(\La), \CS^{\rm{cw}}_{\CX}(\La)$ and $\CS_{\CX}^{\rm{scw}}(\La)$, as proved in this section, implies that they are different  as  exact categories, whenever $\La$ is assumed  not to be semisimple.\\

As an application of our classification of projetive objects and injective objects we can see the exact category $\CS^{\rm{scw}}_{\CX}(\CX)$ is Frobenius whenever the canonical exact category $\CX$ is so. Let us state this observation as the following result.
\begin{corollary}\label{Corollay 3.6Frobenius}
	Assume that the exact category $\CX$ has enough injective.  If the exact category $\CX$ is Frobenius, then  so is $\CS_{\CX}^{\rm{scw}}(\La)$.	
	\end{corollary}

\subsection{Some further exact structures}
   Denote by ${\bf Ex}(\CX)$ an additive category which has the short exact sequences in $\CX$, $0 \rt X_1\st{f}\rt X_2 \st{g}\rt X_3\rt 0$,  sometimes $(X_1 \st{f}\rt X_2\st{g}\rt X_3)$ for short, as its objects. A morphism between two short exact sequences $\epsilon$ and $\eta$,  we mean a triplet $(\phi_1, \phi_2, \phi_3)$ satisfying the following commutative diagram
   	$$\xymatrix{
   	\epsilon: \ \ \ \ \ \ 0 \ar[r] &X_1 \ar[r]\ar[d]^{\phi_1} & X_2 \ar[d]^{\phi_2}
   	\ar[r] &X_3 \ar[d]^{\phi_3}\ar[r] & 0&\\ \eta: \ \ \ \ \ \
   	0 \ar[r] & Y_1 \ar[r]^{s_H} & Y_2
   	\ar[r] & Y_3\ar[r] &0.& \\ 	} 	 $$
   We can view the objects in ${\bf Ex}(\CX)$ as a complex concentrated on degree 1,2 and 3. We denote by $\mathcal{C}^{\rm{b}}(\mmod \La)$ the category of bounded complexes over $\mmod \La.$ As $\CX$ is resolving, by viewing ${\bf Ex}(\CX)$ as a subcategory in $\mathcal{C}^{\rm{b}}(\mmod \La)$, so is ${\bf Ex}(\CX)$ a resolving subcategory. Hence an exact structure on ${\bf Ex}(\CX)$ is inherited form the abelian structure of  $\mathcal{C}^{\rm{b}}(\mmod \La)$. In this way we can consider ${\bf Ex}(\CX)$ as an exact category. Throughout the paper, when we call the additive category ${\bf Ex}(\CX)$ as an exact category we mean this induced exact structure. In similar way to $\CS^{\rm{scw}}_{\CX}(\La)$, we can give another exact structure on ${\bf Ex}(\CX)$ component-wisely as follows. Let
   $$\epsilon: \ \ \ \ 0 \rt (X_1\st{f}\rt X_2 \st{g}\rt X_3)\st{(\phi_1, \phi_2, \phi_3)}\lrt (Z_1\st{f'}\rt Z_2 \st{g'}\rt Z_3)\st{(\phi'_1, \phi'_2, \phi'_3)}\lrt (Y_1\st{f''}\rt Y_2 \st{g''}\rt Y_3)\rt 0$$
   be a short exact sequence in the exact category ${\bf Ex}(\CX)$, or equivalently can be viewed as a short exact sequence in  $\mathcal{C}^{\rm{b}}(\mmod \La)$ with terms in ${\bf Ex}(\CX)$.  The above short exact sequence gives us the following commutative diagram in $\mmod \La$

   $$\xymatrix{& 0 \ar[d] & 0 \ar[d] & 0 \ar[d]&(*) &\\
   	0 \ar[r] & X_1 \ar[d]^f \ar[r]^{\phi_1} & Z_1 \ar[d]^{f'}
   	\ar[r]^{\phi'_1} & Y_1 \ar[d]^{f''} \ar[r] & 0\\
   	0 \ar[r] & X_2\ar[d]^g \ar[r]^{\phi_2} & Z_2
   	\ar[d]^{g'} \ar[r]^{\phi'_2} & Y_2 \ar[d]^{g''} \ar[r] & 0\\
   	0 \ar[r] & X_3 \ar[d] \ar[r]^{\phi_3} & Z_3
   	\ar[d] \ar[r]^{\phi'_3} & Y_3 \ar[d] \ar[r] & 0\\
   	& 0  & 0  & 0 & }
   $$
   Let $\mathbf{\CE x}^{\rm{cw}}(\CX)$ denote the  short exact sequences $\epsilon$ such that all rows $0 \rt X_i \st{\phi}\rt Z_i \st{\phi'}\rt Y_i\rt 0$ in diagram $(*)$ are split. In analogy with $\CE^{\rm{cw}}_{\CX}$ or $\CE^{\rm{scw}}_{\CX}$, we can see the class $\mathbf{\CE x}^{\rm{cw}}(\CX)$ satisfies the axioms of the exact categories. Therefore, we have another exact structure  $({\bf Ex}(\CX), \mathbf{\CE x}^{\rm{cw}}(\CX))$, shown by  ${\bf Ex}^{\rm{cw}}(\CX)$. The assignment an object $(X \st{f}\rt Y)$ in $\CS_{\CX}(\La)$ into the object $0 \rt X \st{f}\rt Y\rt \rm{Cok} \ f \rt 0$ in ${\bf Ex}(\CX)$ defines an equivalence between the additive categories $\CS_{\CX}(\La)$ and  ${\bf Ex}(\CX)$. If we consider $\CS_{\CX}(\La)$ and  ${\bf Ex}(\CX)$ with exact structures
     $\CE^{\rm{scw}}_{\CX}$ and $\mathbf{\CE x}^{\rm{cw}}(\CX)$, respectively, then one can see the assignment preserves the exact structures. Thus there are the equivalences of exact categories $\CS_{\CX}(\La)\simeq {\bf Ex}(\CX)$ and $\CS^{\rm{scw}}_{\CX}(\La)\simeq {\bf Ex}^{\rm{cw}}(\CX)$.

    Based on the above equivalences of exact categories, we obtain the following results for the exact categories ${\bf Ex}(\CX)$ and ${\bf Ex}^{\rm{cw}}(\CX)$.

    \begin{proposition}
    	Let ${\bf X}=(X_1 \st{f} \rt X_2  \st{g}\rt  X_3)$ be an object in ${\bf Ex}(\CX)$. Then ${\bf X}$ is an indecomposable projective   object in the exact category  ${\bf Ex}^{\rm{cw}}(\CX)$, resp  ${\bf Ex}(\CX)$,  if and only if it is isomorphic to either of the following objects
    \begin{itemize}
    	\item [$(1)$] $(X \st{1}\rt X\rt 0)$ for some  indecomposable, resp projective,  object $X$ in $\CX;$
    	\item[$(2)$] $(0 \rt X \st{1}\rt X)$ for some  indecomposable, resp. projective,   object $X$ in $\CX$;
    	\item [$(3)$] $(\Omega_{\La}(X)\hookrightarrow P_X \st{\pi_X}\rt X)$ for some indecomposable  object $X$ in $\CX$.
    \end{itemize}
    Moreover, ${\bf Ex}^{\rm{cw}}(\CX)$, resp  ${\bf Ex}(\CX)$,  has enough projectives.	
    \end{proposition}

\begin{proposition}
	Assume that
	the exact category $\CX$ has enough injectives.
	Let ${\bf X}=(X_1 \st{f} \rt X_2\st{g}\rt X_3)$ be an object in ${\bf Ex}(\CX)$. Then ${\bf X}$ is an indecomposable injective  object in the exact category  ${\bf Ex}^{\rm{cw}}(\CX)$, resp  ${\bf Ex}(\CX)$,  if and only if it is isomorphic to either of the following objects
	\begin{itemize}
		\item [$(1)$] $(X \st{\rm{1}}\rt X\rt 0) $ for some indecomposable, resp. injective,  object $X$ in $\CX;$
		\item [$(2)$]
			 	$(0 \rt X\st{1} \rt X)$ for some indecomposable, resp.  injective, object  in the exact category $\CX$;
		\item [$(3)$]		
		$(X \st{f}\rt I \rt \rm{Cok} \ f)$ for some indecomposable  object $X$ in $\CX,$ where $f$ is a  left minimal  monomorphism and $I $ an injective object in the exact category $\CX$.	
	\end{itemize}
	Moreover, ${\bf Ex}^{\rm{cw}}(\CX)$, resp  ${\bf Ex}(\CX)$, has enough injectives.
\end{proposition}

\section{ equivalences on the stable categories}
 Recall from Subsection \ref{An equivalence}, and also see \cite[Construction 3.1]{H2} for more details, the functor $\Psi_{\CX}:\CS_{\CX}(\La)\rt \mmod \underline{\CX}$ which is assigned to the $\CX$ making a connection between the monomorphism category and modules (or functors) over the stable categories. Since $\mmod \underline{\CX}$ is an abelian category, hence it is reasonable to ask how $\Psi_{\CX}$ becomes an exact functor. With the first and second types of the  exact structures defined on $\CS_{\CX}(\La)$ in previous section, the functor $\Psi_{\CX}$ is not an exact functor in general.  For instance, let $0 \rt A \st{f}\rt B  \st{g}\rt C\rt 0$ be a non-split  short exact sequence in $\CX$. Then the functor $\Psi_{\CX}$ does not preserve the exactness of the following short exact sequence in $\CS_{\CX}(\La)$ 
 {\footnotesize  \[ \xymatrix@R-2pc {  &  ~ A\ar[dd]^{1}~   & A\ar[dd]^{f}~  & 0\ar[dd] \\  0 \ar[r] &  _{ \ \ \ \ } \ar[r]^{1}_{f}  &_{\ \ \ \ \ } \ar[r]^{0}_{g} _{\ \ \ \ \ }&  _{\ \ \ \ \ }\ar[r] & 0. \\ & A & B & C }\]}
 But, by considering the third type of exact structures on $\CS_{\CX}(\La)$, the functor turn into an exact functor. 
  Therefore,  $\Psi_{\CX}:\CS^{\rm{scw}}_{\CX}(\La)\rt \mmod \underline{\CX}$ means that we consider the functor $\Psi_{\CX}$ as an exact functor from the exact category  $\CS^{\rm{scw}}_{\CX}(\La)$ to $\mmod \underline{\CX}$. By enhancing $\Psi_{\CX}$ with such an exactness, we will obtain some new equivalences. The functor $\Psi_{\CX}$ is a functorial approach and generalization of  one of the two functors studied in \cite{E} and \cite{RZ}. 
 
  Recall that an additive functor $F:\CB \rt \CA$ between two exact categories is called {\it exact} provided that it sends conflations to conflations.
 \begin{theorem}\label{Theoerm 5.1}
 	Let $\CX$ be a resolving and  contravariantly finite subcategory of $\mmod \La,$ being, indeed,  our convention on $\CX$ in this paper. Then
 	\begin{itemize}
 		\item [$(1)$] The functor $\Psi_{\CX}:\CS^{\rm{scw}}_{\CX}(\La)\rt \mmod \underline{\CX}$ is full, dense, objective and exact;
 		\item [$(2)$] The functor $\Psi_{\CX}$ induces an equivalence between the stable categories $\underline{\CS^{\rm{cw}}_{\CX}(\La)}$ and $\mmod \underline{\CX}$.
 	\end{itemize}
 \end{theorem}
 \begin{proof}
By \cite[Theorem 3.2]{H2}, see also subsection \ref{An equivalence}, the functor $\Psi_{\CX}$ is full, dense and objective. Just it remains to check that it is an exact  functor. Let
{\footnotesize  \[ \xymatrix@R-2pc {  &  ~ X_1\ar[dd]^{f}~   & Z_1\ar[dd]^{h}~  & Y_1\ar[dd]^{g} \\ \epsilon: \  \ \  \  0 \ar[r] &  _{ \ \ \ \ } \ar[r]^{\phi_1}_{\phi_2}  &_{\ \ \ \ \ } \ar[r]^{\psi_1}_{\psi_2} _{\ \ \ \ \ }&  _{\ \ \ \ \ }\ar[r] & 0 \\ & X_2 & Z_2 & Y_2 }\]}
be a short exact sequence in $\CE^{\rm{scw}}_{\CX}$.  By definition the expended version of the  above sequence in $\mmod \La$ is the following commutative diagram
$$\xymatrix{& 0 \ar[d] & 0 \ar[d] & 0 \ar[d]& \dagger&\\
	0 \ar[r] & X_1 \ar[d]^f \ar[r]^{\phi_1} & Z_1 \ar[d]^h
	\ar[r]^{\psi_1} & Y_1 \ar[d]^g \ar[r] & 0\\
	0 \ar[r] & X_2\ar[d] \ar[r]^{\phi_2} & Z_2
	\ar[d] \ar[r]^{\psi_2} & Y_2 \ar[d] \ar[r] & 0\\
	0 \ar[r] & \rm{Cok}\ f \ar[d] \ar[r]^{\mu_1} & \rm{Cok} \ h
	\ar[d] \ar[r]^{\mu_2} & \rm{Cok}\ g \ar[d] \ar[r] & 0\\
	& 0  & 0  & 0 & }
$$
in $\mmod \La$ with split rows. By applying the Yoneda functor on the diagram $(\dagger)$ we get the following commutative diagram with exact rows and columns in $\mmod \CX$
$$\xymatrix{& 0 \ar[d] & 0 \ar[d] & 0 \ar[d]& 0\ar[d]&  &\\
	0 \ar[r] & (-,X_1) \ar[d] \ar[r]^{(-, f)} & (-, X_2) \ar[d]
	\ar[r] &(-, \rm{Cok}\ f) \ar[d] \ar[r] & F\ar[r]\ar[d]^f&0\\
	0 \ar[r] & (-,Z_1)  \ar[r]^{(-, h)}\ar[d] & (-, Z_2)\ar[d]
	\ar[r] & (-,  \rm{Cok}\ h) \ar[r]\ar[d] & G\ar[r]\ar[d]^g& 0\\0\ar[r] &(-, Y_1)\ar[r]^{(-, g)}\ar[d]&(-,  Y_2)\ar[r]\ar[d]&(-, \rm{Cok} \ g)\ar[d] \ar[r]&H \ar[r]\ar[d]&0 &\\ &0&0&0&0&&
} 	 $$
The exactness of the first three columns except the rightmost columns comes from the fact that the rows in the diagram $(\dagger)$ are split. The  short exact sequence in the rightmost column is exactly the image of the short exact sequence $\epsilon$ under the functor $\Psi_{\CX}.$ This means that $\Psi_{\CX}$ is an exact functor  when we consider $\CS_{\CX}(\La)$ as an exact category via the exact structure $\CE^{\rm{scw}}_{\CX}$. This completes the statement $(1)$. As stated in  \cite[Theorem 3.2]{H2}, since the functor $\Psi_{\CX}$ is full, dense and objective then it induces an equivalence $\CS_{\CX}(\La)/\CU\simeq \mmod \underline{\CX}$, where $\CU$ is an ideal consisting of morphisms which factor through the direct sum of objects in the form $X\st{1}\rt X$ and $0 \rt X$, for some $X \in \CX.$ By the characterization of the indecomposable projective objects in the exact category $\CS^{\rm{cw}}_{\CX}(\La)$, see Proposition \ref{Proposition 3.4}, one can see only by definition  that for any $A$ and $B$ in $\CS^{\rm{cw}}_{\CX}(\La)$, $\CP_{\CE^{\rm{cw}}}(A, B)=\CU(A, B)$. Hence  $\CS_{\CX}(\La)/\CU\simeq \underline{\CS^{\rm{cw}}_{\CX}(\La)}$. So  the second statement is proved.
 \end{proof}
 We intend to give an another variation of the equivalence $\CS_{\CX}(\La)/\CU\simeq \mmod \underline{\CX}$ in the level of the stable categories.
 \begin{theorem}\label{Stableequivalence1}
 	Keep all the notations as above.  The functor $\Psi_{\CX}$ induces an equivalence between  the stable categories of exact categories $\underline{\CS^{\rm{scw}}_{\CX}(\La)} \simeq \underline{\mmod \underline{\CX}}$. Consequently, there is an equivalence of categories
 	$$\CS_{\CX}(\La)/\CV \simeq \underline{\mmod \underline{\CX}},$$
 where $\CV$ is an ideal generated by the object of the form $X\st{1}\rt X, \  0\rt X$ and $\Omega_{\La}(X)\rt P,$ for some $X$ in $\CX.$ 	
 	 \end{theorem}
 	\begin{proof}
 		Since the exact functor $\Psi_{\CX}:\CS^{\rm{scw}}_{\CX}(\La)\rt \mmod \underline{\CX}$, by definition, sends the projective objects in the exact categories $\CS^{\rm{scw}}_{\CX}(\La)$ to the ones in $\mmod \underline{\CX}.$ Hence it induces the functor $\underline{\Psi_{\CX}}:\underline{\CS^{\rm{scw}}_{\CX}(\La)} \rt \underline{\mmod \underline{\CX}}$, which is full and dense as it is induced by a full and dense functor $\Psi_{\CX}.$ So for completing the proof, it suffices to prove that the $\underline{\Psi_{\CX}}$ is faithful. Assume that $(\phi_1, \phi_2): (A \st{f}\rt B)\rt (C\st{g}\rt D)$ is a morphism in $\CS_{\CX}(\La)$  that mapped into zero by the functor $\underline{\Psi_{\CX}}$. Set $\Psi_{\CX}(A\st{f}\rt B):= F, \ \Psi_{\CX}(C \st{g}\rt D):= G$ and $\Psi_{\CX}(\phi_1, \phi_2):=\phi.$ Our assumption implies that there is a module $X$ in $\CX$ such that the following factorization in $\mmod \underline{\CX}$, or in $\mmod {\CX}$ when we consider $\mmod \underline{\CX}$ as a full subcategory of $\mmod \CX$, exists
 		$$\hskip .5cm \xymatrix@-4mm{
 			&(-, \underline{X})\ar @{<-}[dl]_{\eta}\ar[dr]^{\delta}\\
 			F\ar [rr]^{\phi} && G}\hskip .5cm$$
 		Let $0 \rt (-, \Omega_{\La}(X))\rt (-, P)\rt (-, X)\rt (-, \underline{X})\rt 0$ be a projective resolution of $(-, \underline{X})$ in $\mmod \CX.$ It is a standard fact that the morphisms $\eta$ and $\delta$ can be lifted to the corresponding projective resolutions. These facts are recorded in the following commutative diagram

 			$$\xymatrix{    & 0 \ar[r] & (-, A)   \ar@/^1.25pc/@{.>}[dd]^<<<<<<<<{(-, \phi_1)} \ar[d]^{\eta_1}  \ar[r]^{(-, f)}  & (-, B)  \ar@/^1pc/@{.>}[dd]^<<<<<<<<<<<{(-, \phi_2)} \ar[d]^{\eta_2}  \ar[r]  &  (-, \rm{Cok} \ f)  \ar@/^1.2pc/@{.>}[dd]^<<<<<<<{\gamma} \ar[d]^{\alpha}  \ar[r]  & F \ar@/^1.2pc/@{.>}[dd]^{\phi} \ar[d]^{\eta} &   \\
 			&0 \ar[r] & (-, \Omega_{\La}(X)) \ar[r] \ar[d]^{\delta_1} & (-, P) \ar[r] \ar[d]^{\delta_2}  &   (-, X )  \ar[d]^{\beta}  \ar[r]  & (-, \underline{X}) \ar[d]^{\delta}  \\
 			& 0 \ar[r] & (-, C)  \ar[r]^{(-, g)}  & (-, D)  \ar[r]  &(- \rm{Cok} \ g) \ar[r] & G }$$
 		Since $\phi=\delta \circ \eta$ then chain maps $(\delta_1\circ\eta_1, \  \delta_2\circ \eta_2,\  \beta\circ\alpha)$ and $(\widehat{\phi_1}, \ \widehat{\phi_2}, \ \gamma)$ are chain-homotopic. Set $\widehat{\phi_i}:=(-, \phi_i)$, $i=1, 2$. Therefore, the chain map $(\phi_1-\delta_1\circ\eta_1, \ \phi_2-\delta_2\circ \eta_2, \ \gamma-\beta\circ\alpha)$ factors through a projective complex $P$ over $\mmod \CX$. The complex $P$ can be taken a projective precover of the deleted projective resolution of $G$ in the category of complexes over $\mmod \underline{\CX}$. By the Yoneda lemma,  we can write $\eta_i=(-, d_i)$ and $\delta_i=(-, t_i)$, for $i=1, 2$. By returning to $\CS_{\CX}(\La)$ via the Yoneda lemma and using the above observation we can deduce the morphism $(\phi_1-t_1\circ d_1, \ \phi_2-t_2\circ d_2)$ factors through the projective object $(C \st{1}\rt C)\oplus (0 \rt D)$ in the exact category $\CS^{\rm{scw}}_{\CX}(\La)$. Consequently, the morphism $(\phi_1, \phi_2)$ factors through the projective object $(C \st{1}\rt C)\oplus (0 \rt D) \oplus (\Omega_{\La}(X)\rt P)$ of $\CS^{\rm{scw}}_{\CX}(\La)$, so the desired result. The proof is now complete.
 	\end{proof}
 
 \begin{remark}\label{remark 4.3}
 	In a dual manner, whenever the exact category $\CX$ has enough injectives,  one can see that  the functor $\Psi_{\CX}$ also induces an equivalence between the costable categories of exact categories $\overline{\CS^{\rm{scw}}_{\CX}(\La)} \simeq \overline{\mmod \underline{\CX}}$. Let us for making better a reference later denote the induced functors by  $\underline{\Psi_{\CX}}$ and $\overline{\Psi_{\CX}}$ for the stable and costable cases, respectively. 
 \end{remark}

 We recall that  an additive category $\CT$ is called a triangulated  category if there exists an auto equivalence $\Sigma:\CT \rt \CT$, called shift functor, and a class $\bigtriangleup$ of diagrams of the form $A \rt B \rt C \rt \Sigma A$ in $\CT$ satisfying a certain set of axioms. A triangulated functor $F:\CT\rt \CT'$ between two triangulated categories $\CT$ and $\CT'$ is an additive functor that commutes with shift functors and preserves the distinguished triangles. An equivalence of triangulated categories is a  triangulated functor and an equivalence of categories. Diagrams in $\bigtriangleup$ are called distinguished triangles. The basic references for  the triangulated categories are \cite{N} and \cite{V}.\\
Let $(\mathcal{C}, \CE)$ be a Frobenius category. It is known that the stable category $\underline{\mathcal{C}}$ gets a triangulated structure, see \cite[Chapter I, Section 2]{Ha}. The shift functor $\Sigma_{\mathcal{C}}: \underline{\mathcal{C}}\rt \underline{\mathcal{C}}$ is defined such that for $C \in \mathcal{C},$
$$0 \rt C \rt I(C)\rt \Sigma_{\mathcal{C}} C \rt 0,$$
is a short exact sequence in $\CE$ with $I(C)$ an injective object. Distinguished triangles in $\underline{\mathcal{C}}$ are just given by the short exact sequences in $\CE$.\\

 In Corollary \ref{Corollay 3.6Frobenius},  we saw when the exact category $\CX$ is Frobenius then the exact category $\CS^{\rm{scw}}_{\CX}(\La)$ becomes a Frobenius exact category. Assume the exact category $\CX$ is Frobenius, in view of \cite[Chapter I, Section 2]{Ha}, the stable category $\underline{\CX}$ and $\underline{\CS^{\rm{scw}}_{\CX}(\La)}$ get naturally a triangulated structure. On the other hand, since $\underline{\CX}$ is a triangulated category then by
 a classical result from \cite{F}, $\mmod \underline{\CX}$ is a Frobenius (abelian) category.  Consequently, the stable category  $\underline{\mmod \underline{\CX}}$ similarly gets a triangulated structure in a natural way. Hence in this case the functor $\underline{\Psi_{\CX}}$ turns to be a functor between triangulated categories. The natural question may arise here whether $\underline{\Psi_{\CX}}$ is a triangulated functor. In the sequel,  we will show that the answer to this question is in the affirmative.
 \begin{theorem}\label{Theorem 5.3}
 	Keep all the notations as above. Assume the exact category $\CX$ is Frobenius. The functor $\Psi_{\CX}$ induces a triangle  equivalence between  the triangulated categories $\underline{\CS^{\rm{scw}}_{\CX}(\La)} \simeq \underline{\mmod \underline{\CX}}$.	
 \end{theorem}
\begin{proof}
	Set $\Sigma_{\underline{\CX}}$ and $\Sigma^{\rm{scw}}$ for the shift functors in the triangulated categories $\underline{\mmod \underline{\CX}}$ and $\underline{\CS^{\rm{scw}}_{\CX}(\La)}$, respectively. According to \cite[Lemma in page 23]{Ha}, we need to define an invertible natural transformation $\alpha: \underline{\Psi_{\CX}}\Sigma^{\rm{scw}} \rt \Sigma_{\underline{\CX}}\ \underline{\Psi_{\CX}}$. To this end, by definition, we can  fix for each $F \in \mmod \underline{\CX}$ a short exact sequence
	$$0 \rt F \rt (-, \underline{X})\rt \Sigma_{\underline{\CX}} F \rt 0 \ \ \ \ \ (*)$$
	for some $X$ in $\CX.$ Clearly $X$ depends on the  given functor $F$. Note  since $\mmod \underline{\CX}$ is Frobenius, then  $(-, \underline{X})$ is an injective functor in $\mmod \underline{\CX}$.  Also, fix a minimal  projective resolution in $\mmod \CX$ as the following for  $F$	
	$$0 \rt(-, A_F)\rt (-, B_F)\rt (-, C_F)\rt F\rt 0.$$	
	Using the Horseshoe lemma implies the following commutative diagram in $\mmod \CX$ with splitting  columns except the rightmost one
		$$\xymatrix{& 0 \ar[d] & 0 \ar[d] & 0 \ar[d]& 0\ar[d]& &\\
		0 \ar[r] & (-,A_F) \ar[d] \ar[r] & (-, B_F) \ar[d]
		\ar[r] &(-, C_F) \ar[d] \ar[r] & F\ar[r]\ar[d]^f&0\\
		0 \ar[r] & (-,A_F\oplus A_{\Sigma_{\underline{\CX}}F})  \ar[r]\ar[d] & (-, B_F\oplus B_{\Sigma_{\underline{\CX}}F})\ar[d]
		\ar[r] & (-,  C_F\oplus C_{\Sigma_{\underline{\CX}}F}) \ar[r]\ar[d] & (-, \underline{X})\ar[r]\ar[d]^g& 0\\0\ar[r] &(-, A_{\Sigma_{\underline{\CX}}F})\ar[r]\ar[d]&(-,  B_{\Sigma_{\underline{\CX}}F})\ar[r]\ar[d]&(-, C_{\Sigma_{\underline{\CX}}F})\ar[d] \ar[r]&\Sigma_{\underline{\CX}}F \ar[r]\ar[d]&0 &\\ &0&0&0&0&&
	} 	 $$
	The exact sequence in the  middle row  of diagram above is indeed a projective resolution of $(-, \underline{X})$ in $\mmod \CX$. We know the exact sequence $0 \rt (-, \Omega_{\La}(X))\rt (-, P)\rt (-, X)\rt (-, \underline{X})\rt 0$ acts as a minimal projective resolution of $(-, \underline{X})$ in $\mmod \underline{\CX}$. Hence the deleted projective resolution of $(-, \underline{X})$ appeared in the second row of diagram above is isomorphic to  the direct sums of the following complexes
	$$0 \rt (-, \Omega_{\La}(X)) \rt (-, P)\rt (-, X)\rt 0$$
	$$0 \rt (-, D) \st{(-, 1)}\rt (-, D)\rt 0 \rt 0, \ \ \  \text{for some} \ \ \ D  \ \text{in} \ \ \CX,$$
	$$0 \rt 0 \rt (-, L) \st{(-, 1)}\rt (-, L) \rt 0, \ \ \  \text{for some} \ \ \ L  \ \text{in} \ \ \CX.$$
	Here the first $(-, D)$ and $(-, L)$ from the left side are settled 	at degrees $-2$ and $-1$, respectively. Returning to $\CS_{\CX}(\La)$ via the Yoneda Lemma and in view of the above facts, we obtain the following short exact sequence
{\footnotesize  \[ \xymatrix@R-2pc {  &  ~ A_F\ar[dd]~   & G\ar[dd]~  & A_{\Sigma_{\underline{\CX}}F}\ar[dd]\\ \epsilon: \  \ \  \  0 \ar[r] &  _{ \ \ \ \ } \ar[r] &_{\ \ \ \ \ } \ar[r] _{\ \ \ \ \ }&  _{\ \ \ \ \ }\ar[r] & 0 \\ & B_F & H& B_{\Sigma_{\underline{\CX}}F} }\]}
in $\CE^{\rm{scw}}_{\CX}$, where $(G \rt H)\simeq (\Omega_{\La}(X)\rt P)\oplus(D\st{1}\rt D)\oplus (0 \rt L)$. Hence the middle term in  $(\epsilon)$	 by the characterization given in Proposition \ref{Proposition 3.4} is an injective object in the exact category $\CS^{\rm{scw}}_{\CX}(\La)$. Hence we may define the shift functor $\Sigma^{\rm{scw}}$ by the  end term of the above short exact sequence. We can observe that $ \underline{\Psi_{\CX}}\Sigma^{\rm{scw}}(A_F \rt B_F) = \Sigma_{\underline{\CX}}\ \underline{\Psi_{\CX}}(A_F\rt B_F)$.  Now we define  $\alpha_{(A_F\rt B_F)}=\rm{Id}_{(A_F\rt B_F)}$. On the other hand, since any object $(X \st{g}\rt Y)$ in $\CS_{\CX}(\La)$ is isomorphic to $(A_F\rt B_F)$ for some $F$ in $\underline{\CS^{\rm{scw}}_{\CX}(\La)}$, then one can in an obvious way extend the $\alpha$ to all  objects. In fact, applying the Yoneda functor to the short exact sequence
$$0 \rt X \st{g}\rt Y \rt \rm{Cok} \ g \rt 0$$ induced by the object $(X\st{g}\rt Y)$, we get the following sequence in $\mmod \CX$
$$0 \rt (-, X)\rt (-, Y)\rt (-, \rm{Cok}\ g) \rt F \rt 0.$$
Similar to the argument above by comparing the above projective resolution of $F$ with the minimal one,  we can see the isomorphism $(X \st{g}\rt Y)\simeq (A_F \rt B_F)$ in $\underline{\CS^{\rm{scw}}_{\CX}(\La)}$.	
\end{proof}

\begin{example}\label{Example 4.5}
Let $n>0$ and $k$ be a field. Let $k[x]$ be the polynomial ring in one variable $x$ with coefficients in $k$ and $\La_n=k[x]/(x^n)$. Let $\Pi_n$ be the preprojective algebra of type $\mathbb{A}_n$. The preprojective algebras were introduced by Gelfand and Ponomare \cite{GP}.  In view of  \cite{RZ},  َ$\mmod \underline{\rm{mod}}\mbox{-}\La_n\simeq \mmod \Pi_{n-1}$. Hence by Theorem \ref{Theorem 5.3}, we have triangle equivalence $\underline{\rm{mod}}\mbox{-}\Pi_{n-1}\simeq \underline{\CS^{\rm{scw}}(\La_n)}$. Therefore, this triangle equivalence might be helpful to transfer results for submodule categories to preprojective algebras. For instance,  by \cite[Corollary 6.5]{RS2} and our result,  the Auslander-Reiten translation, as defined in \cite{Ha}, in the triangulated category $\underline{\rm{mod}}\mbox{-}\Pi_{n}$ is invariant under the sixth power.  
\end{example}

\section{Auslander-Reiten-Serre duality}
In this section we will show that three kinds of exact structures on $\CS_{\CX}(\La)$ have almost split sequences, and equivalently Auslander-Reiten-Serre duality, if we assume further $\CX$ is functorily finite.\\
Let us begin  with some backgrounds about almost split sequences in an exact categories. Let $(\mathcal{C}, \CE)$ be an exact category. Recall that a morphism $v\colon E\to Y$ in $\mathcal{C}$ is \emph{right almost split} if it is not a retraction and each $f\colon Z\to Y$ which is not a retraction factors through $v$. Dually, a morphism $u\colon X\to E$ in $\mathcal{C}$
is \emph{left almost split} if it is not a section and each $f\colon X\to Z$ which is not a section factors through $u$. A  short exact sequence  $\delta\colon 0 \rt X \xrightarrow{u} E \xrightarrow{v} Y \rt 0$ in $\CE$ is an \emph{almost split sequence} if $u$ is left almost split and $v$ is right almost split.  Since $\delta $ is unique up to isomorphism for $X$ and $Z$, we may write $X=\tau_{\mathcal{C}}Y$ and $Y=\tau^{-1}_{\mathcal{C}}X$. We call $X$ the Auslander-reiten translation of $Y$ in $\mathcal{C}$.  A non-zero object $X \in \mathcal{C}$ is said to be {\it endo-local} if $\rm{End}_{\mathcal{C}}(X)$ is local. We say that $\mathcal{C}$ has right almost split sequences if for any endo-local non projective object $Y$, there exits an almost split short exact sequence $\delta\colon 0 \rt X \xrightarrow{u} E \xrightarrow{v} Y \rt 0$ in $\CE$. Dually, we can define the notion of having left almost sequences for the exact category $\mathcal{C}$.   We say that $\mathcal{C}$ has {\it almost split sequences} if it has right and left almost split sequences. Assume further  that $\mathcal{C}$ is a {\it $k$-linear category}, that is, for each $X, Y \in \mathcal{C}, $ $\Hom_{\mathcal{C}}(X, Y)$ is equipped with a structure of $k$-modules and the composition of morphisms of $\mathcal{C}$ is $k$-linear. For example, taking $\CX$ a subcategory of $\mmod \La$, as assumed $\La$ is a $k$-algebras, then $\CX$ gets canonically  a $k$-linear category. Whenever $\mathcal{C}$ is a $k$-linear category in a natural way $\Ext^1_{\mathcal{C}}(X, Y)$ gets a structure of $k$-modules. In this case we call that $\mathcal{C}$ is {\it Ext-finite} if for any $X, Y \in \mathcal{C}$, $\Ext^1_{\mathcal{C}}(X, Y)$ is a finitely generated $k$-module. An Auslander-Reiten-Serre (ARS-) duality is a pair $(\tau, \sigma)$ of an equivalence functor $\sigma_{\mathcal{C}}: \underline{\mathcal{C}}\rt \overline{\mathcal{C}}$ and a binatural isomorphism
$$\eta_{X, Y}:\Hom_{\underline{\mathcal{C}}}(X, Y)\simeq D\Ext^1_{\mathcal{C}}(Y, \sigma_{\mathcal{C}} X) \ \ \ \text{for any}  \ \ X, Y \in \mathcal{C}.$$
If $\mathcal{C}$ has an ARS-duality $\sigma_{\mathcal{C}}$, then it is unique up to isomorphism,  so using the notation  $\sigma_{\mathcal{C}}$ for the ARS dualities is well-defined up to isomorphism. This notation is helpful to distinguish  ARS-dualities  whenever we are dealing with different exact categories simultaneously. Due to \cite[Theorem 3.6]{INY}  or \cite{J},  there is a close connection between the existence of almost split sequences in $\mathcal{C}$ and ARS-duality.

\begin{theorem}\label{ARS duality}
	Let $\mathcal{C}$ be a $k$-linear Ext-finite Krull-Schmidt exact category with  enough projectives and injectives.. Then the following conditions are equivalent.
	\begin{itemize}
		\item [$(1)$] $\mathcal{C}$ has almost split sequences;
		\item [$(2)$] $\mathcal{C}$ has an Auslander-Reiten-Serre duality;
		\item [$(3)$] The stable category $\underline{\mathcal{C}}$ is a dualizing $k$-variety;
		\item [$(4)$] The costable category $\overline{\mathcal{C}}$ is a dualizing $k$-variety.
	\end{itemize}	
\end{theorem}
Let us mention the above theorem in \cite[Theorem 3.6]{INY} is given for more general setting of extriangulated categories. Recently, the class of extriangulated categories was introduced in \cite{NP} as a simultaneous generalization of exact categories and triangulated categories.\\ 
Let us add this fact in the case  $\mathcal{C}$ of  theorem above, for any indecomposable non-projective $X$ in $\mathcal{C}$, $\tau_{\mathcal{C}} X\simeq \sigma_{\mathcal{C}}X$, if we remove all the injective summand appearing in a decomposition of $\sigma_{\mathcal{C}}X$ into  indecomposable objects.
\begin{lemma}\label{Lemma 4.2}
	Let $\CX$ be a resolving and functorialy finite subcategory of $\mmod \La.$ Assume the exact category $\CX$ has enough injectives. Then the following statments hold.
	\begin{itemize}
		\item[$(1)$]	
		for any injective object $F$ in $\mmod \underline{\CX}$, there is some $X \in \CX$ such that $F \simeq \Ext^1_{\La}(-, X)\mid_{\CX};$
		\item [$(2)$] If $\Ext^1_{\La}(-, X)\mid_{\CX} \simeq \Ext^1_{\La}(-, Y)\mid_{\CX}$, for any $X$ and $Y$ in $\CX$, then $X\simeq Y$ in $\overline{\CX}.$
	\end{itemize}
\end{lemma}
\begin{proof}
	Thanks to \cite[Theorem 2.4]{AS}, since by our assumption $\CX$ is functorily finite,  $\CX$ has almost split sequences. Hence there exists by Theorem \ref{ARS duality} the equivalence $\sigma_{\CX}:\underline{\CX}\rt \overline{\CX}$ such that for any module $X \in \CX$, there is an isomorphism of the following functors
	$$\underline{\rm{Hom}}_{\La}(\sigma^{-1}_{\CX}X, -)\simeq D\Ext^1_{\La}(-, X)|_{\CX} \ \ \ \ \ \ \ \ \  (*)$$
	Note that here $\overline{\CX}$ is the costable category of the exact category $\CX$,  not to be necessarily the full subcategory of injectively stable category $\overline{\rm{mod}}\mbox{-}\La$ consisting of  modules in $\CX$. Since $\CX$ is closed under direct summands then $\underline{\CX}$ has split idempotents. This follows that any projective functor in $\mmod (\underline{\CX})^{\op}$ is in the form $\underline{\rm{Hom}_{\La}}(Y, -)$ for some $Y$ in $\CX.$ We will use this fact to prove our claim. Assume $F$ is an injective object in $\mmod \underline{\CX}$. Lemme \ref{Lemma 3.19} implies the existence of a monomorphim $0 \rt F \rt \Ext^1(-, A)|_{\CX}$ for some $A \in \CX$. Since $F$ is injective then it is a direct summand of $\Ext^1(-, A)|_{\CX}$, so we can write $\Ext^1(-, A)|_{\CX}= F\oplus G.$ Using $(*)$ for the object $A$ yields $DF$  is a direct summand of $\underline{\rm{Hom}}_{\La}(A, -)$. Now the fact follows that $DF\simeq \underline{\rm{Hom}}_{\La}(B, -)$ for some $B.$ Again using $(*)$ for $B$, we have $\underline{\rm{Hom}}_{\La}(B, -)\simeq D\Ext^1_{\La}(-, \sigma_{\CX}B)$. This gives $F\simeq \Ext^1_{\La}(-, \sigma_{\CX}B)$, as desired. The second assertion follows from the ARS-duality in conjunction with the Yoneda Lemma.
\end{proof}
We recall from \cite[Section 2]{AR1}, a Krull-Schmidt Hom-finite $k$-linear category $\mathcal{C}$ is a {\it dualizing $k$-variety} if the standard $k$-duality $D:\Mod \mathcal{C} \rt \Mod(\mathcal{C}^{\rm{op}}), F \mapsto D\circ F$ induces a duality $D:\mmod \mathcal{C}\rt \mmod \mathcal{C^{\rm{op}}}$.
\begin{proposition}\label{Almost-Split-Secondtype}
	Let $\CX$ be a functorialy finite resolving subcategory of $\mmod \La$. Assume the exact category $\CX$ has enough injectives.
	Then the exact category $\CS^{\rm{cw}}_{\CX}(\La)$ has almost split sequences.
\end{proposition}
\begin{proof}
	Since $\CX$ is functorialy finite then the stable category $\underline{\CX}$ becomes a dualizing $k$-variety. Indeed, let $F \in \mmod \underline{\CX}$. Then $DF$ can be regarded as a finitely presented functor in $\mmod \CX^{\rm{op}}$ which vanish on $\rm{prj}\mbox{-}\Lambda$. Hence $DF \in \mmod \underline{\CX}^{\rm{op}}$. Conversely, we can show that $DF' \in \mmod \underline{\CX}$ for any $F'$ in $\mmod \underline{\CX}^{\rm{op}}$. Thanks to \cite[Proposition 2.6]{AR1},	 $\mmod \underline{\CX}$ is a dualizing $k$-variety. Due to the equivalence given in Subsection \ref{An equivalence}, we have an equivalence between the additive quotient category $\CS_{\CX}(\La)/\CU$ and abelian category $\mmod \underline{\CX}$. We recall $\CU$ is an idempotent ideal generated by the object of the form $X\st{1}\rt X$
	and $0\rt X$, where $X$ runs through the objects in $\CX$. In view of the characterization of indecomposable projective objects in the exact category $\CS^{\rm{cw}}_{\CX}(\La)$, one can deduce the quotient category $\CS_{\CX}(\La)/\CU$ is nothing else than the stable category $\underline{\CS^{\rm{cw}}_{\CX}(\La)}$ of the exact category $\CS^{\rm{cw}}_{\CX}(\La)$. Hence by the equivalence we can see that $\underline{\CS^{\rm{cw}}_{\CX}(\La)}$ is a dualizing $k$-variety. Now  Theorem \ref{ARS duality}
	follows that the exact category $\CS^{\rm{cw}}_{\CX}(\La)$ has almost split sequences. We are done.
\end{proof}
We use the above  to strengthen the results  given in \cite[Theorem 2.5]{RS2} and \cite[Proposition 2.5]{H2}  in which are proved $\CS_{\CX}(\La)$ has almost split sequences for the cases $\CX=\mmod \La$ and $\CX=\rm{Gprj}\mbox{-}\La$.
\begin{proposition}\label{Almost-Split First and third}
	Let $\CX$ be a functorialy finite resolving subcategory of $\mmod \La$. Assume the exact category $\CX$ has enough injectives.
	Then the following statements hold true.
	\begin{itemize}
		\item [$(1)$] The exact category $\CS_{\CX}(\La)$ has almost split sequences;
		\item [$(2)$] The exact category $\CS^{\rm{scw}}_{\CX}(\La)$ has almost split sequences.
	\end{itemize}
\end{proposition}
\begin{proof}
	$(1)$ 	We only show that for each indecomposable non-projective object $(A \st{f}\rt B)$ in the exact category $\CS_{\CX}(\La)$ there is an almost split sequence ending at $(A \st{f}\rt B)$. The proof of $\CS_{\CX}(\La)$ has left almost split sequences is in a dual manner.
	By the characterization of indecomposable projectives given in Proposition \ref{projective objects in monomorphism}, we observe  $(A \st{f}\rt B)$ is not isomorphic to neither $P \st{1}\rt P$ nor $(0 \rt P)$, for some indecomposable projective module $P$. First, if it is isomorphic to either $(X\st{1}\rt X), (0 \rt X)$ or $\Omega_{\La}(X)\rt P$ for some indecomposable non-projective $X$, then by \cite[Lemma 6.3]{H2} there is an almost split sequence in $\CS_{\CX}(\La)$ as needed. Assume $(A \st{f}\rt B)$ is not isomorphic to those indecomposable objects in $\CS_{\CX}(\La)$. Hence it is a non-projective indecomposable object in the exact category $\CS^{\rm{cw}}(\La)$. It follows from Proposition \ref{Almost-Split-Secondtype} that there is an almost split sequence as the following in the exact category $\CS^{\rm{cw}}_{\CX}(\La)$ {\footnotesize  \[ \xymatrix@R-2pc {  &  ~ C\ar[dd]^{g}~   & Y\ar[dd]^{h}~  & A\ar[dd]^{f} \\ \epsilon: \  \ \  \  0 \ar[r] &  _{ \ \ \ \ } \ar[r]^{\phi_1}_{\phi_2}  &_{\ \ \ \ \ } \ar[r]^{\psi_1}_{\psi_2} _{\ \ \ \ \ }&  _{\ \ \ \ \ }\ar[r] & 0 \\ & D & Z & B }\]}
	ending at $(A \st{f}\rt B)$. The short exact  sequence $\epsilon$ clearly is in $\CE_{\CX}$, By definition one can see easily the short exact  sequences $\epsilon$ acts also as an almost split sequence in the exact category $\CS_{\CX}(\La)$. $(2)$  By the observation given in  \cite[Construction 5.4]{H2} the sequence $\epsilon$ also lies in $\CE^{\rm{scw}}_{\CX}$. Hence again by definition we can see it also act as an almost split sequence in the exact category $\CS^{\rm{scw}}_{\CX}(\La)$ which ends at $(A \st{f}\rt B)$. For the convenience of the reader let us explain why $\epsilon$ belongs to $\CS^{\rm{scw}}_{\CX}(\La)$. Since we know $\epsilon$ is in $\CE^{\rm{cw}}_{\CX}$ then the rows $0 \rt C \st{\phi_1}\rt Y \st{\psi_1}\rt A\rt 0$ and $0 \rt D\st{\phi_2} \rt Z \st{\psi_2} \rt B\rt 0$ are split. Hence it suffices to show that the induced short exact sequence obtained by getting cokernel
	$$0 \rt \rm{Cok} \ g \st{\mu_1}\rt \rm{Cok} \ h \st{\mu2} \rt \rm{Cok} \ f \rt 0$$
	is split. To this do, consider the morphism $[\sigma_1~~\sigma_2]:(\Omega_{\La}(\rm{Cok}\ f)\rt P)\rt (A\st{f}\rt B)$	obtaining with the following commutative diagram
	$$\xymatrix{
		0 \ar[r] &\Omega_{\La}(\rm{Cok}\ f) \ar[d]^{\sigma_1} \ar[r] & P_{\rm{Cok} \ f} \ar[d]^{\sigma_2}
		\ar[r] &\rm{Cok}\ f \ar@{=}[d]\ar[r] & 0&\\
		0 \ar[r] & A  \ar[r]^{f} & B
		\ar[r] & \rm{Cok}\ f \ar[r] &0.& \\ 	} 	 $$
	But $[\sigma_1~~\sigma_2]$ is not retraction, otherwise it leads  to  $B$ to be a projective. This contradicts  the fact that $(A \st{f}\rt B)$ is not isomorphic to an indecomposable object in the form $\Omega_{\La}(X)\rt P$ for some indecomposable module $X$. Hence $(\sigma_1,~~\sigma_2)$ factors through $(\psi_1,~~\psi_2)$, as $(\psi_1, \psi_2)$ is a right almost split morphism, and consequently $\rm{Id}_{\rm{Cok}\ f}$ factors through $\mu_2$. But this means that $\mu_2$ is split epimorphism, so the result follows. The proof is complete
\end{proof}

Denote by $\underline{\rm{SHom}}(-, -), \ \underline{\rm{SHom}}^{\rm{cw}}(-, -)$ and $\underline{\rm{SHom}}^{\rm{scw}}(-, -)$ the Hom-space in the stable categories  $\underline{\CS_{\CX}(\La)}, \ \underline{\CS^{\rm{cw}}_{\CX}(\La)}$ and $\underline{\CS^{\rm{scw}}_{\CX}(\La)}$,  respectively. Also, $\rm{SExt}^1(-, -), \ \rm{SExt}^1_{\rm{cw}}(-, -)$ and $\rm{SExt}^1_{\rm{scw}}(-, -)$, respectively,  denote the group of extensions in  the exact categories $\CS_{\CX}(\La), \ \CS^{\rm{cw}}_{\CX}(\La)$ and $\CS^{\rm{scw}}_{\CX}(\La)$,  similar notations of whose counterpart for  $\CF_{\CX}(\La)$ can be defined only by changing the   letter ``S'' to ``F''.\\

Combining Theorem \ref{ARS duality} and Proposition \ref{Almost-Split-Secondtype} and \ref{Almost-Split First and third} imply the following theorem.
\begin{theorem}\label{ARSDualityFOrthreeKindExatStrtures}
	Let $\CX$ be a functorialy finite resolving subcategory of $\mmod \La$. Assume the exact category $\CX$ has enough injectives. Then three types of exact structures defined on  $\CS_{\CX}(\La)$ have  Auslander-Reiten-Serre duality. More precisely, let $X=(X_1\st{f}\rt X_2)$ and $Y=(Y_1\st{g} \rt Y_2)$ be two objects in the additive category $\CS_{\CX}(\La)$. Then
	\begin{itemize}
		\item [$(1)$] There is an equivalence functor $\sigma_{\CS\CX}:\underline{\CS_{\CX}(\La)}\rt \overline{\CS_{\CX}(\La)}$ and a binatural isomorphism
		$\eta_{X, Y}:\underline{\rm{SHom}}(X, Y)\simeq D\rm{SExt}^1(Y, \sigma_{\CS\CX} X) .$
		\item [$(2)$]There is an equivalence functor $\sigma^{\rm{cw}}_{\CS\CX}:\underline{\CS^{\rm{cw}}_{\CX}(\La)}\rt \overline{\CS^{\rm{cw}}_{\CX}(\La)}$ and a binatural isomorphism
		$\eta^{\rm{cw}}_{X, Y}:\underline{\rm{SHom}}^{\rm{cw}}(X, Y)\simeq D\rm{SExt}^1_{\rm{cw}}(Y, \sigma^{\rm{cw}}_{\CS\CX} X) .$
		\item [$(3)$]There is an equivalence functor $\sigma^{\rm{scw}}_{\CS\CX}:\underline{\CS^{\rm{scw}}_{\CX}(\La)}\rt \overline{\CS^{\rm{scw}}_{\CX}(\La)}$ and a binatural isomorphism
		$\eta^{\rm{scw}}_{X, Y}:\underline{\rm{SHom}}^{\rm{scw}}(X, Y)\simeq D\rm{SExt}^1_{\rm{scw}}(Y, \sigma^{\rm{scw}}_{\CS\CX} X) .$
	\end{itemize}
\end{theorem}
By our conventions in the first of the section, showing the ARS-duality by $\sigma_{\mathcal{C}}$ for an exact category $\mathcal{C}$, so we should use the unique ARS-duality of the exact categories $\CS_{\CX}(\La), \ \CS^{\rm{cw}}_{\CX}(\La)$ and $\CS^{\rm{scw}}_{\CX}(\La)$ by $\sigma_{\CS_{\CX}(\La)}, \ \sigma_{\CS^{\rm{cw}}_{\CX}(\La)}$ and $\sigma_{\CS^{\rm{scw}}_{\CX}(\La)}$ respectively, but for brevity we show them by those of presented in the theorem . 
By the above result it would be interesting to compute $\sigma_{\CS\CX}$ for different cases of $\CX$. This is done in \cite{RS2} for the case $\CX=\mmod \La$ by introducing the notion ``$\rm{Mimo}$''.\\

\begin{remark}\label{remark 5.6} Note that  proofs of Proposition \ref{Almost-Split First and third} and \ref{Almost-Split-Secondtype} show when assume $X$ is not a projective object in the exact categories of $\CS^{\rm{scw}}_{\CX}(\La)$, then the Auslander-Reiten translation of $X$ in each of three different exact structures on $\CS_{\CX}(\La)$ are isomorphic to each other. Therefore,  by considering $\sigma_{\CS\CX}X, \sigma^{\rm{cw}}_{\CS\CX}X$ and $\sigma^{\rm{scw}}_{\CS\CX}X$ with no injective summand objects of the relevant exact categories, we have the isomorphism $\sigma_{\CS\CX}X\simeq  \sigma^{\rm{cw}}_{\CS\CX}X\simeq\sigma^{\rm{scw}}_{\CS\CX}X$
	in the additive category $\CS_{\CX}(\La)$.
\end{remark}

Let $\CX$ be the same as in Theorem \ref{ARSDualityFOrthreeKindExatStrtures}. Since $\underline{\CX}$ is a $k$-dualizing variety then there exists the ARS-duality $\sigma_{\CS\CX}:\underline{\CX}\rt \overline{\CX}.$ 
As an application of theorem above, we can give a  description of the $e^1(\sigma_{\CS\CX}X)$ and $e^2(\sigma_{\CS\CX}X)$ via the equivalence $\sigma_{\CX}:\underline{\CX}\rt \overline{\CX}$. Here $e^1$ and $e^2$ denote the evaluation functors from $\rm{H}(\La)$ to $\mmod \La$ in the following way: $e^1(X\st{f}\rt Y):=X$ and $e^2(X\st{f}\rt Y):=Y$ for any $(X\st{f}\rt Y)$ in $\rm{H}(\La)$.  It is well-known that these functors both admit left and right adjoints. The subscriptions $\la$ and $\rho$ are reserved respectively to denote the left and right adjoint functors.
\begin{corollary}(Compare with \cite[Corollary 7.3(3)]{RS2})
	Let $\CX$ be a functorialy finite resolving subcategory of $\mmod \La$. Let $X=(X_1\st{f}\rt X_2)$ be an object in $\CS_{\CX}(\La)$. Then there are isomorphisms $e^1(\sigma_{\CS\CX}X)\simeq \sigma_{\CX}X_2$ and $e^2(\sigma_{\CS\CX}X)\simeq \sigma_{\CX}\rm{Cok} \ f$ in the costable category $\overline{\CX}$.
\end{corollary}
\begin{proof}
	We begin to  prove the first isomorphism. Let $A$ be an arbitrary object in $\CX$. By the functorial isomorphism given in Theorem \ref{ARSDualityFOrthreeKindExatStrtures}(1) for $Y=(A \st{1}\rt A)$ we have the functorial  isomorphism
	$$\underline{\rm{SHom}}(X, Y)\simeq D\rm{SExt}^1(Y, \sigma_{\CS\CX} X)  \ \  \ \  (*)$$
	The adjoint pairs $(e^2, e^2_{\rho})$ and $(e^1_{\la}, e^1)$ respectively imply the following functorial isomorphisms
	$$\underline{\rm{SHom}}(X, Y)\simeq \underline{\Hom}_{\La}(X_2, A) \ \ \ \ \text{and}$$
	$$\rm{SExt}^1(Y, \sigma_{\CS\CX} X) \simeq \Ext^1_{\La}(A, e^1(\sigma_{\CS\CX}X)).$$
	Hence in view of $(*)$  we get the following functorial isomorphism
	$$\underline{\Hom}_{\La}(X_2, A)\simeq D\Ext^1_{\La}(A, e^1(\sigma_{\CS\CX}X)) \ \ \ \ (**)$$
	Moreover, since the exact category $\CX$ has the  ARS-duality, say $\sigma_{\CX}$,  then we have the following functorial isomorphism
	$$\underline{\Hom}_{\La}(X_2, A)\simeq D\Ext^1_{\La}(A, \sigma_{\CX}X_2) \ \ \ \ (***)$$
	Putting together
	$(**)$ and $(***)$ implies the isomorphism of  functors $\Ext^1_{\La}(-, e^1(\sigma_{\CS\CX}X)) \simeq \Ext^1_{\La}(-, \sigma_{\CX}X_2)$. Now Lemma \ref{Lemma 4.2} completes the first part. According to Remark \ref{RemarkSToF}, there are mutually inverse exact equivalences $\CK\rm{
		er}$ and $\mathcal{C}\rm{ok}$ between the canonical  exact categories $\CS_{\CX}(\La)$ and $\CF_{\CX}(\La)$. Let $A$ and $B$  be in  $\CF_{\CX}(\La)$. Therefore, since they are exact functors then they induce mutually inverse equivalences between their (co)stable categories and also functorial isomorphisms between their groups of extensions.  Consider the following functorial isomorphisms obtained by those facts:
	\begin{align*}
	\underline{\rm{FHom}}(A, B)&\simeq \underline{\rm{SHom}}(\CK\rm{er} \ A, \  \CK\rm{er} \ B) \\
	&\simeq  D\rm{SExt}^1(\CK\rm{er} \ B, \ \sigma_{\CS\CX}\circ \CK\rm{er} \ A) \\
	& \simeq D\rm{FExt}^1( B, \  \mathcal{C}\rm{ok}\circ\sigma_{\CS\CX}\circ \CK\rm{er} \ A)\\
	\end{align*}
	Hence $\mathcal{C}\rm{ok}\circ\sigma_{\CS\CX}\circ \CK\rm{er}$ gives an ARS duality for the exact category $\CF_{\CX}(\La)$. Put $A=\mathcal{C}\rm{ok} \ X$ and $B=(Y\st{1}\rt Y)$ in the above functorial  isomorphism. Then we get
	$$\underline{\rm{FHom}}(\mathcal{C}\rm{ok} \ X, Y) \simeq D\rm{FExt}^1(Y, \  \mathcal{C}\rm{ok}\circ\sigma_{\CS\CX} X).$$
	Applying the adjoint pairs and using the same argument as given in the first of the proof, we obtain the following functorial isomorphism
	$$\underline{\Hom}_{\La}(e^2(\mathcal{C}\rm{ok} \ X), Y)\simeq D\Ext^1_{\La}(Y, e^1(\mathcal{C}\rm{ok}\circ\sigma_{\CS\CX} X)) \ \ \ \ (\dagger).$$
	But $e^1(\mathcal{C}\rm{ok} \ X)=\rm{Cok} \ f $ and $e^1(\mathcal{C}\rm{ok}\circ\sigma_{\CS\CX} X)= e^2(\sigma_{\CS\CX} X)$. By considering these equalities in the $(\dagger)$, we reach
	$$\underline{\Hom}_{\La}(\rm{Cok} \ f, Y)\simeq D\Ext^1_{\La}(Y, e^2(\sigma_{\CS\CX} X)) \ \ \ \ (\dagger\dagger).$$
	Analogous to the first part, since $\CX$ has ARS duality then there exists the following functorial isomorphism
	$$\underline{\Hom}_{\La}(\rm{Cok} \ f, Y)\simeq D\Ext^1_{\La}(Y, \sigma_{\CX}\rm{Cok} \ f) \ \ \ \ (\dagger\dagger\dagger)$$
	Lemma \ref{Lemma 4.2}(2) together with $(\dagger\dagger)$ and $(\dagger\dagger\dagger)$ imply the second isomorphism in the statement.
\end{proof}
In above corollary a local description of $\sigma_{\CS\CX}$ is given without having an explicit formula of it,  in contrast to \cite[Corollary 7.3(3)]{RS2} where an explicit computation of the Auslander-Reiten translation in $\CS(\La)$ is used to prove it. \\

By applying the result of Sections $4$ and $5$, We will give a relationship between ARS-duality for the exact category $\CS^{\rm{scw}}_{\CX}(\La)$ and $\mmod \underline{\CX}$. To show Hom-space in the stable category $\underline{\mmod \underline{\CX}}$ and group extension in $\mmod \underline{\CX}$, in the next result,  we will use $\underline{\rm{Hom}}_{\underline{\CX}}(-, -)$ and $\Ext^1_{\underline{\CX}}(-, -)$, respectively.
\begin{proposition}
	Let $\CX$ be a functorialy finite resolving subcategory of $\mmod \La$. Assume the exact category $\CX$ has enough injectives. Consider the following equivalences
	$$\underline{\mmod \underline{\CX}}\st{(\underline{\Psi_{\CX}})^{-1}}\lrt \underline{\CS^{\rm{scw}}_{\CX}(\La)}\st{\sigma^{\rm{scw}}_{\CS\CX}}\lrt \overline{\CS^{\rm{scw}}_{\CX}(\La)}\st{\overline{\Psi_{\CX}}}\lrt \overline{\mmod \underline{\CX}}, $$
	where $(\underline{\Psi_{\CX}})^{-1}$ and $\overline{\Psi_{\CX}}$ defined in Remark \ref{remark 4.3}.
	Then the composition $\overline{\Psi_{\CX}}\circ \sigma^{\rm{scw}}_{\CS\CX} \circ (\underline{\Psi_{\CX}})^{-1}$ is an ARS-duality for $\mmod \underline{\CX}$.	 
\end{proposition} 
 
\begin{proof}
	Let $F$ and $G$ be in $\mmod \underline{\CX}$. Then, there are  following natural isomorphisms of groups
	
	\begin{align*}
	\underline{\rm{Hom}}_{\underline{\CX}}(F, G)&\simeq\underline{\rm{SHom}}^{\rm{scw}}((\underline{\Psi_{\CX}})^{-1}(F), \  (\underline{\Psi_{\CX}})^{-1}(G)) \\
	&\simeq D\rm{SExt}^1_{\rm{scw}}((\underline{\Psi_{\CX}})^{-1}(G), \ \sigma^{\rm{scw}}_{\CS\CX}\circ (\underline{\Psi_{\CX}})^{-1}(F) ) \\
	&\simeq D\rm{Ext}^1_{\underline{\CX}}(G , \ \underline{\Psi_{\CX}}\circ \sigma^{\rm{scw}}_{\CS\CX}\circ (\underline{\Psi_{\CX}})^{-1}(F)).\\
	\end{align*}
The first isomorphism 	follows from the equivalence $(\underline{\Psi_{\CX}})^{-1}$ and the second isomorphism from Theorem \ref{ARSDualityFOrthreeKindExatStrtures}. Sending an extension of $\sigma^{\rm{scw}}_{\CS\CX}\circ (\underline{\Psi_{\CX}})^{-1}(F)$ by $(\underline{\Psi_{\CX}})^{-1}(G)$ in the exact category $\CS_{\CX}^{\rm{scw}}(\La)$ , say, $0\rt \sigma^{\rm{scw}}_{\CS\CX}\circ (\underline{\Psi_{\CX}})^{-1}(F)\rt X\rt (\underline{\Psi_{\CX}})^{-1}(G)\rt 0$ into the extension $0\rt \Psi_{\CX} (\sigma^{\rm{scw}}_{\CS\CX}\circ (\underline{\Psi_{\CX}})^{-1}(F))\rt \Psi_{\CX}(X)\rt \Psi_{\CX}((\underline{\Psi_{\CX}})^{-1}(G))\rt 0$ by applying the exact functor $\Psi_{\CX}$ gives us an isomorphism of groups from $\rm{SExt}^1_{\rm{scw}}((\underline{\Psi_{\CX}})^{-1}(G), \ \sigma^{\rm{scw}}_{\CS\CX}\circ (\underline{\Psi_{\CX}})^{-1}(F) )$ to 
$\rm{Ext}^1_{\underline{\CX}}(G , \ \underline{\Psi_{\CX}}\circ \sigma^{\rm{}scw}_{\CS\CX}\circ (\underline{\Psi_{\CX}})^{-1}(F))$, so the last isomorphism holds in the above by applying the duality $D$.  	
\end{proof}
 The above result in view of Remark \ref{remark 5.6} provides a nice connection between the Auslander-Reiten translation in $\CS_{\CX}(\La)$ and $\mmod \underline{\CX}$, see also \cite[Section 5]{H2}.

\section{Equivalence of singularity   categories}

In Section 3 we showed  by defining the new exact structure (third type) on $\CS_{\CX}(\La)$, the functor $\Psi_{\CX}$ becomes an exact functor . Enhancing $\CS_{\CX}(\La)$ with this new exact structure,  the  functor  $\Psi_{\CX}$ induces a triangle functor between  the corresponding derived. Then we will show that  the induced triangle functor gives the triangle equivalences in Theorem \ref{Theorem 1.1}. Then,  we will give  an interesting  application for singularly equivalences of Morita type.\\

Since the functor $\Psi_{\CX}$ is additive then by applying terms by terms over complexes of $\CS^{\rm{scw}}_{\CX}(\La)$,  we can obtain a triangle functor $\K^{\rm{b}}\Psi_{\CX}:\K^{\rm{b}}(\CS^{\rm{scw}}_{\CX}(\La))\rt \K^{\rm{b}}(\mmod \underline{\CX})$. On the other hand,  since $\Psi_{\CX}$ preserves the short exact sequences in the exact category $\CS^{\rm{scw}}_{\CX}(\La)$, therefore the acyclic complexes $\rm{Ac}^{\rm{b}}(\CS^{\rm{scw}}_{\CX}(\La))$ are mapped into $\rm{Ac}^{\rm{b}}(\mmod \underline{\CX})$. Hence $\K^{\rm{b}}\Psi_{\CX}$ induces the triangle functor $\mathbb{D}^{\rm{b}}\Psi_{\CX}:\D^{\rm{b}}(\CS^{\rm{scw}}_{\CX}(\La))\rt \D^b(\mmod \underline{\CX})$. Before proceeding to the proofs of our results, we need to give some facts which are known  for abelian categories, but we require to state them in the setting of exact categories. We refer to \cite{P} for a recent  reference of such backgrounds we need. Let $\mathcal{C}$ be a  weakly split idempotent small exact category. By the definition, we define the $\Ext$ group $\Ext_{\mathcal{C}}^n(X, Y)=\Hom_{\D^{\rm{b}}(\CS^{\rm{scw}}_{\CX}(\La))}(X, Y[n])$, where an object in the exact category viewed as a complex concentrated in degree zero.
The Ext groups with negative numbers
are zero, i.e., $\Ext^n_{\mathcal{C}}(X,Y)=0$ for $n<0$.
The natural morphism $\Hom_{\mathcal{C}}(X,Y)\rt \Ext^0_{\mathcal{C}}(X,Y)$ is
an isomorphism; in other words, the functor $\mathcal{C}\rt \D^b(\mathcal{C})$
is fully faithful.
The Ext groups with positive numbers $n>0$ are computed by
the following Yoneda construction.
Consider the class $\rm{Yon}^n_{\mathcal{C}}(X,Y)$ of all $n$-extensions $A^\bu=(0 \rt Y\rt A^{n-1}\rt \cdots \rt A^0\rt X\rt 0)$ in the exact category $\mathcal{C}$. Viewing an $n$-extension as a complex, it is just an acyclic complex over $\mathcal{C}$.
Two $n$-extensions $A$ and $B\in\rm{Yon}^n_{\mathcal{C}}(X,Y)$ are said to be equivalent
if there exists a third element $C\in\rm{Yon}^n_{\mathcal{C}}(X,Y)$ and two morphisms
of $n$-extensions $C\rt A$ and $C\rt B$,  both are identity on the induced morphisms of  the terms $Y$ and $X$. A morphism between two $n$-extensions is only a chain map of complexes whenever the $n$-extensions considered  as complexes.
This is indeed an equivalence relation and let  $\rm{Yon}^n_{\mathcal{C}}(X,Y)$, with the same symbol, denote the set of all the (Yoneda) equivalences which is canonically bijective to $\rm{Ext}^n_{\mathcal{C}}(X,Y)$.
In fact, the assignment  of the Yoned equivalence class $[0 \rt Y \rt A^{n-1}\rt \cdots \rt A^0\st{g}\rt  X\rt 0]$ into the equivalence class of roof $X \st{s}\leftarrow Z \st{f}\rt Y[n]$ makes an isomorphism of groups. Here $Z=(0 \rt Y \rt A^{n-1}\rt \cdots \rt A^0\rt 0)$, that is a complex with $A^0$ at degree zero, $s$ is a  chain map such that in all degrees is zero except the 0-th degree that is $g$, and $f$ is a chain map with zero in all degrees except the $n$-th degree that is the identity $\rm{Id}_Y$.\\
Let $X$ and $Y$ be in the exact category $\CS^{\rm{scw}}_{\CX}(\La)$. Consider the following composition $A_{[X,Y]}$ of homomorphisms of groups
{\tiny $$ \rm{Yon}^n_{\CS^{\rm{scw}}_{\CX}(\La)}(X, Y) \st{\simeq}\rt\Hom_{\D^{\rm{b}}(\CS^{\rm{scw}}_{\CX}(\La))}(X, Y[n])\rt \Hom_{\D^{\rm{b}}(\mmod \underline{\CX})}(\Psi_{\CX}(X), \Psi_{\CX}(Y)[n]) \st{\simeq} \rt \rm{Yon}^n_{\mmod\underline{\CX}}(\Psi_{\CX}(X), \Psi_{\CX}(Y)),$$}
where the first and last isomorphisms are just the canonical isomorphisms introduced in the above for any exact category, which here specialized for the exact categories $\CS^{\rm{scw}}_{\CX}(\La)$ and $\mmod \underline{\CX}$, and the middle groups homomorphism is induced by the functor $\D^{\rm{b}}\Psi_{\CX}$ . Note that by the definition for any $Z \in \CS^{\rm{scw}}_{\CX}(\La)$, we have $\D^{\rm{b}}\Psi_{\CX}(Z)=\Psi_{\CX}(Z)$. By following only the definitions of the groups homomorphisms involved in the above composition we observe the Yoneda equivalence class of an $n$-extension, say $0 \rt Y \rt A^{n-1}\rt \cdots \rt A^0\rt X\rt0,$ in $\rm{Yon}^n_{\CS^{\rm{scw}}_{\CX}(\La)}(X, Y)$ is mapped into the Yoneda equivalence class of an $n$-extension in $\rm{Yon}^n_{\mmod\underline{\CX}}(\Psi_{\CX}(X), \Psi_{\CX}(Y))$, just by  applying terms by terms of the functor $\Psi_{\CX}$, i.e., $0 \rt \Psi_{\CX}(Y) \rt \Psi_{\CX}(A^{n-1})\rt \cdots \rt \Psi_{\CX}(A^0)\rt \Psi_{\CX}(X)\rt0.$ \\
 
We start with the following vital  lemma.
\begin{lemma}\label{Lemma6.1}
For any $X$ and $Y$ in $\CS^{\rm{scw}}_{\CX}(\La)$, the group morphism $A_{[X, Y]}:  \rm{Yon}^n_{\CS^{\rm{scw}}_{\CX}(\La)}(X, Y)\rt \rm{Yon}^n_{\mmod\underline{\CX}}(\Psi_{\CX}(X), \Psi_{\CX}(Y))$, defined in the above, is  surjective.
\end{lemma}
\begin{proof}
Take an $n$-extension $F^{\bu}=(0 \rt \Psi_{\CX}(Y)\st{d^n}\rt F^{n-1}\rt \cdots \rt F^0\st{d^0}\rt \Psi_{\CX}(X)\rt 0)$. Assume that $K^i=\Ker(F^i\st{d^i}\rt F^{i-1})$, $0 \leq i \leq n-1$, and set $F^{-1}:= \Psi_{\CX}(X)$. By our notation $K^{n-1}=\Psi_{\CX}(Y)$ and also set $K^{-1}:= \Psi_{\CX}(X)$. For any $0 \leq i \leq n-1$, the $n$-extension gives the short exact sequence $\epsilon_i: \ 0 \rt K^i\rt F^i\rt K^{i-1}\rt 0$ in $\mmod \underline{\CX}$. Since $\Psi_{\CX}$ is dense we can find $A^i \in \CS_{\CX}(\La)$, and for any $-1 \leq i \leq n-1$, such that $\Psi_{\CX}(A^i)=K^i$, for $i= -1, n-1$, take $K^{-1}=X$ and $K^{n-1}=Y,$ respectively. By using the horseshoe lemma for each the obtained  short exact sequence in $\mmod \underline{\CX}$, we can construct a short exact sequence $\eta_i: \ 0 \rt A^i \rt M^i \rt M^{i-1}\rt 0$ in $\CE^{\rm{scw}}_{\CX}$ which is mapped into $\epsilon_i$  by applying $\Psi_{\CX}$. By splicing the short exact sequences $\eta_i$ together, we then obtain  an $n$-extension $G^{\bu}$ of $X$ and $Y$ in $\CS^{\rm{scw}}_{\CX}(\La)$. By construction we observe $A_{[X, Y]}([G^{\bu}])=[F^{\bu}]$, so the result.   
\end{proof}
An immediate consequence of the above lemma is the induced  groups homorphism
$$\Hom_{\D^{\rm{b}}(\CS^{\rm{scw}}_{\CX}(\La))}(X^{\bu}, Y^{\bu}) \rt \Hom_{\D^b(\mmod \underline{\CX})}(\D^{\rm{b}}\Psi_{\CX}(X^{\bu}), \D^{\rm{b}}\Psi_{\CX}(Y^{\bu}))$$ by $\mathbb{D}^{\rm{b}}\Phi_{\CX}$
is surjective. That we will use this fact in the next result. It might be possible to show directly  the surjectivity of the induced groups homorphism, but  using the Yoneda construction, as used to prove Lemma \ref{Lemma6.1}, helps to prove the fact easier.

A bounded complex over an additive category $ِ\CA$ is called  of length at most $n$ if there are at most $n$ terms to be non-zero.
\begin{proposition}\label{FullProposition}
	The triangle functor $\mathbb{D}^{\rm{b}}\Phi_{\CX}:\D^{\rm{b}}(\CS^{\rm{scw}}_{\CX}(\La))\rt \D^b(\mmod \underline{\CX})$ is full.
\end{proposition}
\begin{proof}
We must prove for each $X^{\bu}$ and $Y^{\bu}$ in $\D^{\rm{b}}(\CS^{\rm{scw}}_{\CX}(\La))$ the induced  groups homomorphism
$$\Hom_{\D^{\rm{b}}(\CS^{\rm{scw}}_{\CX}(\La))}(X^{\bu}, Y^{\bu}) \rt \Hom_{\D^b(\mmod \underline{\CX})}(\D^{\rm{b}}\Psi_{\CX}(X^{\bu}), \D^{\rm{b}}\Psi_{\CX}(Y^{\bu}))$$ by $\mathbb{D}^{\rm{b}}\Phi_{\CX}$
	is surjective. First, as said above,  Lemma \ref{Lemma6.1} implies the induced groups homomorphism is surjective for any $X^{\bu}$ and $Y^{\bu}$ with length at most one. Now, by using  a standard manner, one can   extend to any $X^{\bu}$ and $Y^{\bu}$. For the completeness we will give a proof here. First by induction, suppose the claim is true for any $X^{\bu}$ and $Y^{\bu}$ with length at most $n$. Assume complexes $X^{\bu}$ and $Y^{\bu}$ with length at most $n+1$ are given. Then there exists a triangle 
	$$ Y_2^{\bu}\rt Y_1^{\bu}\rt Y^{\bu}\rt Y^{\bu}_2[1]$$
	with $Y_2^{\bu}$ of length at most $n$ and $Y_1^{\bu}$ of length at most one. Applying
	the cohomological functors  $\Hom_{\D^{\rm{b}}(\CS^{\rm{scw}}_{\CX}(\La))}(X^{\bu},-)$ and $\Hom_{\D^b(\mmod \underline{\CX})}(\D^{\rm{b}}\Psi_{\CX}(X^{\bu}),-)$ yields the following diagram with exact rows: 
	{\tiny	$$\xymatrix{
		\Hom(X^{\bu},Y_2^{\bu})\ar[r]\ar[d]^{\phi_0} & \Hom(X^{\bu},Y_1^{\bu})\ar[r]\ar[d]^{\phi_1} & \Hom(X^{\bu},Y^{\bu}) \ar[d]^{\phi_2}
		\ar[r] &\Hom(X^{\bu},Y_2^{\bu}[1]) \ar[d]^{\phi_3}\ar[r] & \Hom(X^{\bu},Y_1^{\bu}[1])\ar[d]^{\phi_4}&\\ 
		\Hom(\overline{X^{\bu}}, \overline{Y_2^{\bu}}) \ar[r] & \Hom(\overline{X^{\bu}},\overline{Y_2^{\bu}}) \ar[r]^{s_H} & \Hom(X^{\bu},Y^{\bu})
		\ar[r] & \Hom(\overline{X^{\bu}},\overline{Y_2^{\bu}[1]})\ar[r] &\Hom(\overline{X^{\bu}},\overline{Y_1^{\bu}[1]}).& \\ 	} 	 $$}
	For abbreviation the symbols $\D^{\rm{b}}(\CS^{\rm{scw}}_{\CX}(\La)), \D^b(\mmod \underline{\CX})$ are suppressed,  and also for any $A^{\bu}$ in $\D^{\rm{b}}(\CS^{\rm{scw}}_{\CX}(\La))$ set $\overline{A^{\bu}}:=\D^{\rm{b}}\Psi_{\CX}(A^{\bu})$. The all above vertical groups homomorphisms are induced by the functor $\D^{\rm{b}}\Psi_{\CX}$. Hence by the  Five Lemma,  the surjectivity of $\phi_2$, which is our task, follows from surjectivity of $\phi_i$, $i=0, 1, 3, 4.$ To prove surjectivity of $\phi_1$ and $\phi_4$, it suffices to use this fact which the induced groups homeomorphism

	$$\Hom_{\D^{\rm{b}}(\CS^{\rm{scw}}_{\CX}(\La))}(A^{\bu}, B^{\bu}) \rt \Hom_{\D^b(\mmod \underline{\CX})}(\D^{\rm{b}}\Psi_{\CX}(A^{\bu}), \D^{\rm{b}}\Psi_{\CX}(B^{\bu}))$$
	 by $\mathbb{D}^{\rm{b}}\Psi_{\CX}$
	is surjective, when $A^{\bu}$ is any bounded complex and $B^{\bu}$ is a complex of length at most one. To show this fact we need to use again induction on the length of $A^{\bu}$ and using similar argument as we are doing for current induction. For $\phi_0$ and $\phi_3$, first consider the following triangle
	$$ X_2^{\bu}\rt X_1^{\bu}\rt X^{\bu}\rt X^{\bu}_2[1]$$
	with $X_2^{\bu}$ of length at most $n$ and $X_1^{\bu}$ of length at most one. Then by applying the homological functors  $\Hom_{\D^{\rm{b}}(\CS^{\rm{scw}}_{\CX}(\La))}(-, Y^{\bu}_2)$ and $\Hom_{\D^b(-, \mmod \underline{\CX})}(-, \D^{\rm{b}}\Psi_{\CX}(Y_2^{\bu}))$ on the triangle we obtain a commutative diagram as above. Then by using the fact stated in the above and inductive hypothesis we can deduce $\phi_0$ is surjective. Similarly, we can prove the surjectivity of $\phi_3$ by applying the the homological functors  $\Hom_{\D^{\rm{b}}(\CS^{\rm{scw}}_{\CX}(\La))}(-, Y^{\bu}_2[1])$ and $\Hom_{\D^b(\mmod \underline{\CX})}(-, \D^{\rm{b}}\Psi_{\CX}(Y_2^{\bu}[1]))$ on the triangle. So we are done.	
\end{proof}
The fullness of  $\mathbb{D}^{\rm{b}}\Psi_{\CX}$, as proved in the previous proposition, and its densness for the stalk complexes imply the functor to be dense. 
\begin{proposition}\label{denseProp6.3}
	The triangle functor $\mathbb{D}^{\rm{b}}\Psi_{\CX}:\D^{\rm{b}}(\CS^{\rm{scw}}_{\CX}(\La))\rt \D^b(\mmod \underline{\CX})$ is dense.
\end{proposition}
\begin{proof}	Let $F^{\bu}$ be a bounded complex over $\mmod \underline{\CX}$ of length at most $n$. We proceed by induction on $n$. For $n=1$, denseness of $\Psi_{\CX}$ implies the case. Assume the assertion holds true for any complex of length at most $n-1$. There exists the following triangle 	
		$$ F_2^{\bu}
		\st{f}\rt F_1^{\bu}\rt F^{\bu}\rt F^{\bu}_2[1]$$
	with $F_2^{\bu}$ of length at most $n-1$ and $F_1^{\bu}$ of length at most one. By inductive hypothesis and the first step of induction, there are complexes $X^{\bu}$ and $Y^{\bu}$ such that $\D^{\rm{b}}\Psi_{\CX}(X^{\bu})\simeq F_2^{\bu}$ and $\D^{\rm{b}}\Psi_{\CX}(Y^{\bu}) \simeq F_1^{\bu}$. Invoking to Proposition \ref{FullProposition}, we have a morphism $\phi$ in $\Hom_{\D^{\rm{b}}(\CS^{\rm{scw}}_{\CX}(\La))}(X^{\bu}, Y^{\bu})$ such that $\D^{\rm{b}}\Psi_{\CX}(\phi)=f$. The morphism $\phi$ can be fitted in a triangle as the following in $\D^{\rm{b}}(\CS^{\rm{scw}}_{\CX}(\La))$
		$$X^{\bu}\st{\phi}\rt Y^{\bu}\rt Z^{\bu}\rt X^{\bu}[1].$$
	Applying the functor $\D^{\rm{b}}\Psi_{\CX}$ on the above triangle gives us a triangle isomorphic to the first triangle in the above, hence $\D^{\rm{b}}\Psi_{\CX}(Z^{\bu})=F^{\bu}$. So we are done.	
\end{proof}

Now we are ready to give a derived version of the equivalence in \cite{E,RZ}.\\

Before stating the theorem we need recall the notion of thick subcategory of a triangulated category.  Let $\CT$ be a triangulated category with a shift functor $\Sigma$. A triangulated subcategory $\CL$ of $\CT$ is called  {\it thick}, if it is closed under direct summands. The smallest full  thick subcategory  of $\CT$ containing a class $\CL$ of objects is denoted by $< \CL>$.

We have the following constructions.

\begin{construction}\label{Cons}(see \cite[Lemma 3.3]{K}) Given a class  $\CL$  of objects of a triangulated category $\CT$, we take $\overline{\CL}$ to be the class of all $\Sigma^iX$ with $X \in \CL$ and $i \in \Z$. Define a full subcategories $\lan \CL\ran_n$, for $n>0$, inductively as follows.
	\begin{itemize}
		\item ${\lan \CL \ran}_1$ is the  subcategory of $\mathcal{T}$ consisting of all direct summands of objects of $\overline{\CL}$.
		\item For $n>1$, suppose that $\overline{\CL}_n$ is the class of objects $X$ occurring in a triangle
		$$Y \rt X \rt Z \rt \Sigma Y$$   with
		$Y \in {\lan \CL \ran}_{i}$ and $Z\in {\lan \CL \ran}_j$ such that $i,j <n$.
		Let $\lan \CL\ran_n$ denote  the full subcategory of $\CT$ formed by all direct summands of objects of $\overline{\CL}_n$.
		\\It can be easily checked that $< {\CL}> =\bigcup_{n\in \N}{\lan \CL \ran}_n$.
	\end{itemize}
\end{construction}
\begin{theorem}\label{Theorem 6.5}
	Let $\CX$ be a resolving and contravariantly finite subcategory $\CX$ of $\mmod \La.$ Let $\CU$ denote the smallest thick subcategory of $\D^{\rm{b}}(\CS^{\rm{scw}}_{\CX}(\La))$ containing all complexes concentrated in degree zero with terms of the forms $(X\st{1}\rt X)$ and $(0\rt X)$, where $X$ runs through all objects in $\CX$.  The triangle functor $\mathbb{D}^{\rm{b}}\Psi_{\CX}$ induces the following equivalence of triangulated categories
	$$\D^{\rm{b}}(\CS^{\rm{scw}}_{\CX}(\La))/\CU \simeq \D^b(\mmod \underline{\CX}).$$	 
\end{theorem}
\begin{proof} Throughout of the proof we say that an object $A\st{f}\rt B$ has the desired form if it is isomorphic to either $X\st{1}\rt X$ or $0\rt X$, for some $X$ in $\CX.$
	We know by Propositions \ref{FullProposition} and \ref{denseProp6.3} the functor $\mathbb{D}^{\rm{b}}\Psi_{\CX}$ is dense and full. It is known in this case the functor $\mathbb{D}^{\rm{b}}\Psi_{\CX}$ induces a triangle equivalence between the Verdier quotient  $\D^{\rm{b}}(\CS^{\rm{scw}}_{\CX}(\La))/\Ker \ \mathbb{D}^{\rm{b}}\Psi_{\CX}  \simeq \D^b(\mmod \underline{\CX})$. The $\Ker \ \mathbb{D}^{\rm{b}}\Psi_{\CX}$ is a thick subcategory of $\D^{\rm{b}}(\CS^{\rm{scw}}_{\CX}(\La))$ consisting of all objects $F^{\bu}$ such that $\mathbb{D}^{\rm{b}}\Psi_{\CX}(F^{\bu})=0$. To get the equivalence it is enough to show that $\CU=\Ker \ \mathbb{D}^{\rm{b}}\Psi_{\CX}$. By definition it is clear that  all complexes concentrated in degree zero with terms of the desired forms  are in  $\Ker \ \mathbb{D}^{\rm{b}}\Psi_{\CX}$. Hence  $\CU$ is contained in  $\Ker \ \mathbb{D}^{\rm{b}}\Psi_{\CX}$. Let $A^{\bu}$ be a bounded complex in $\Ker \ \mathbb{D}^{\rm{b}}\Psi_{\CX}$. Without of loss generality we can assume that 
	
		{\footnotesize  \[ \xymatrix@R-2pc {&&A_0\ar[dd]^{f_0}~& A_1\ar[dd]^{f_1}~ &      & A_{n-1}\ar[dd]^{f_{n-1}}~ & A_n\ar[dd]^{f_n} \\&A^{\bu}: \  \ \  \  0 \ar[r]&_{\ \ \ \ \ }\ar[r]^{\psi^0_1}_{\psi^0_2} _{\ \ \ \ \ }&_{ \ \ \ \ \ \ } \ar[r] & \cdots  \ar[r]  &_{\ \ \ \ \  \ \  \ \ \  \ \ 			
			} \ar[r]^{\psi^{n-1}_1}_{\psi^{n-1}_2} _{\ \ \  }&  _{\ \ \ \ \ }\ar[r] & 0,
				 \\&&B_0&B_1 && B_{n-1} & B_n }\]}
	
concentrated at degrees $0$ to $n$. In view of Construction \ref{Cons}, it has to be shown that  the complex $A^{\bu}$ can be obtained inductively from the stalk complexes as described in the assertion. In this proof we say that a monomorphism $(A\st{f}\rt B)$ in $\CS_{\CX}(\La)$ is minimal, if the induced projective resolution $0 \rt (-, A)\rt (-, B)\rt (-, \rm{Cok}\ f)\rt F\rt 0$, set $F:=\Psi_{\CX}(A\st{f}\rt B)$, is minimal in $\mmod \CX,$  equivalently, there is no any direct summand isomorphic to an object in the form of $X\st{1}\rt X$ and $0\rt X$ in $\CS_{\CX}(\La)$.   Therefore, by our assumption, the image of $A^{\bu}$ subject to $\mathbb{D}^{\rm{b}}\Psi_{\CX}$
$$\mathbb{D}^{\rm{b}}\Psi_{\CX}(A^{\bu}): \ \ \ \ 0 \rt \Psi_{\CX}(A_0\st{f_0}\rt B_0)\rt  \cdots \rt \Psi_{\CX}(A_n\st{f_n}\rt B_n)\rt 0$$
is an exact complex over $\mmod \underline{\CX}$. Let $(C\st{g}\rt D)$ denote the direct sum of all direct summands of $A_n\st{f_n}\rt B_n$ isomorphic to  the desired forms. Let $A_1^{\bu}$ be a complex concentrated in degree $n$ with the term $(C\st{g}\rt D)$. Then, there is the following triangle in $\D^{\rm{b}}(\CS^{\rm{scw}}_{\CX}(\La))$ 
$$A^{\bu}_1\rt A^{\bu}\rt A^{\bu}_2\rt A^{\bu}_1[1] \ \ \ \ \ \ \ \ \ \ \  *$$	
According to the triangle above, we need to show that $A^{\bu}_2$ belongs to $\CU$. But the $n$-th term of $A^{\bu}_2$	is minimal. Hence we can replace $A^{\bu}_2$ by $A^{\bu}_1$, so $(A_n\st{f_n}\rt B_n)$ may be assumed to be minimal. If $n=0$,  we are done clearly. Hence we may assume $n >0.$ Now we prove that the morphism $(\psi^{n-1}_1, \psi^{n-1}_2)$ is an admissible epimorphism in $\CE^{\rm{scw}}_{\CX}$. Set for any $0 \leqslant i \leqslant n$, $C_i:=\rm{Cok} \ f_i, \ F_i:= \Psi_{\CX}(A_i\st{f_i}\rt B_i) $ and $d_i:=\Psi_{\CX}((\psi^i_1, \psi^i_2))$, where $(\psi^n_1, \psi^n_2):=(0, 0)$. The morphism $(\psi^{n-1}_1, \psi^{n-1}_2)$ induces the following commutative diagram in $\mmod \CX$
	$$\xymatrix{& & & &\dagger& \\
	0 \ar[r] & (-,A_{n-1}) \ar[d]^{\widehat{\psi^{n-1}_1}} \ar[r]^{\widehat{f_{n-1}}} & (-, B_{n-1}) \ar[d]^{\widehat{\psi^{n-1}_2}}
	\ar[r] &(-, C_{n-1}) \ar[d]^{\widehat{r_{n-1}}} \ar[r]^{s_{n-1}} & F_{n-1}\ar[r]\ar[d]^{d_{n-1}}&0\\
	0 \ar[r] & (-,A_n)  \ar[r]^{\widehat{f_n}} & (-,B_n )
	\ar[r] & (-, C_n) \ar[r] & F_n\ar[r]\ar[d]& 0\\ &  &   &  &0 & &} $$
We recall that in the above we let $\widehat{g}:= (-, g)$ for any morphism $g$ in $\mmod \La.$ Let $K_{n-1}:= \Ker \ d_{n-1}$, and consider 
the minimal projective resolution $0 \rt (-, X_{n-1})\st{\widehat{k_{n-1}}}\rt (-, Y_{n-1})\rt (-, Z_{n-1})\rt K_{n-1}\rt 0$ in $\mmod \CX.$ Due to the minimal property and  in view of the  diagram $(\dagger)$, the following commutative diagram exists
	$$\xymatrix{
  & (-, Z_{n-1})\ar[d]
	\ar[r] & (-, C_{n-1}) \ar[r]^{\widehat{r_{n-1}}}\ar[d]^{s_{n-1}} & (-, C_n)\ar[d]& \\0\ar[r]&K_{n-1}\ar[r]\ar[d]&F_{n-1}\ar[d] \ar[r]&F_n \ar[r]\ar[d]&0 &\\ &0&0&0&&
} 	 $$
such that it satisfies the following conditions:
\begin{itemize}
	\item [$(1)$] There is a decomposition $C_{n-1}=Z'_{n-1}\oplus C'_n\oplus D_{n-1}.$
	\item[$(2)$] The restriction of $s$ on the direct summand $(-, D_{n-1})$ is zero, and also the restriction of $r_{n-1}$ on $ Z'_{n-1}\oplus D_{n-1}$ is zero.
	\item[$(3)$] The induced sequence by the top row, $$0 \rt (-, Z_{n-1})\rt (-, Z'_{n-1}\oplus C'_n)\rt (-, C_n)\rt 0$$ is split. 	
\end{itemize}
The existence of the decomposition in $(1)$ follows from this fact that the induced morphism $(-, Z_{n-1}\oplus C_n) \rt F_{n-1}$ provides a projective cover of $F_{n-1}$ in $\mmod \CX.$\\

	Repeating the same argument for the  following induced diagram of $(\dagger)$
		$$\xymatrix{
		0 \ar[r] & (-,A_{n-1}) \ar[d]^{\widehat{\psi^{n-1}_1}} \ar[r]^{\widehat{f_{n-1}}} & (-, B_{n-1}) \ar[d]^{\widehat{\psi^{n-1}_2}}
		\ar[r] &F \ar[d] \ar[r] & 0\\
		0 \ar[r] & (-,A_n)  \ar[r]^{\widehat{f_n}} & (-,B_n )
		\ar[r] & G \ar[r]\ar[d] &0\\&&&0&&  } $$

	 leads to  the following commutative diagram
	$$\xymatrix{&&&&0\ar[d]&\\0 \ar[r]& (-, X_{n-1})\ar[r]^{\widehat{k_{n-1}}}\ar[d]^{\widehat{u_{n-1}}}&(-, Y_{n-1})\ar[r]\ar[d]^{\widehat{w_{n-1}}} &(-, Z_{n-1})\ar[r]\ar[d] &K_{n-1}\ar[d]\ar[r]&0  \\
	0 \ar[r] & (-,A_{n-1}) \ar[d]^{\widehat{\psi^{n-1}_1}} \ar[r]^{\widehat{f_{n-1}}} & (-, B_{n-1}) \ar[d]^{\widehat{\psi^{n-1}_2}}
	\ar[r]^{\widehat{g_{n-1}}} &(-, C_{n-1}) \ar[d] \ar[r]^{s_{n-1}} & F_{n-1}\ar[r]\ar[d]^{d_{n-1}}&0\\
	0 \ar[r] & (-,A_n)  \ar[r]^{\widehat{f_n}} & (-,B_n )
	\ar[r] & (-, C_n) \ar[r] & F_n\ar[r]\ar[d]& 0\\ &  &   &  &0 & &} $$
with the following conditions:
\begin{itemize}
	\item [$(a)$] There are decomposition $B_{n-1}=B_n'\oplus Y'_{n-1}\oplus E_{n-1}$ and $A_{n-1}=A_n'\oplus X'_{n-1}\oplus H_{n-1}$.
	\item[$(b)$] The restriction of $\psi^{n-1}_2$,  resp. $\psi^{n-1}_1$,  on $Y'_{n-1}\oplus E_{n-1}$, resp. $X'_{n-1}\oplus H_{n-1}$, is zero. Moreover, 
 the  restriction   of $f_{n-1}$, resp. $g_{n-1}$, on $H_{n-1}$, resp. $E_{n-1}$,
 
  is a split momorphism, resp. epimorphism,   to $E_{n-1}$, resp. $D_{n-1}$.
		\item[$(c)$] The morphism $u_{n-1}$, resp. $w_{n-1}$,  does not have an image to $H_{n-1}$, resp. $E_{n-1}$.  
		\item [$(d)$] The induced sequence $0 \rt X_{n-1}\rt X'_{n-1}\oplus A'_n\rt A_n \rt 0$, resp. $0 \rt Y_{n-1}\rt Y'_{n-1}\oplus B_n \rt B_n \rt 0,$ is split.
\end{itemize}	
The condition $(b)$ implies the momorphism $(A_{n-1}\st{f_{n-1}}\rt B_{n-1})$ has the direct summand $(H_{n-1}\st{f_{n-1}\mid}\rt E_{n-1})$ with the desired form, and by condition $(c)$ the direct summand has no image under the morphism $(\psi^{n-1}_1, \psi^{n-1}_2)$ in $(A_n\st{f_n}\rt B_n)$. The latter fact implies the existence of a triangle as $(*)$ such that $A^{\bu}_1$ is a stalk complex concentrated at degree $n-1$ with the term $(H_{n-1}\st{f_{n-1}\mid}\rt E_{n-1})$ and $A^{\bu}_2$ a bounded complex in $\Ker \ \mathbb{D}^{\rm{b}}\Psi_{\CX}$ satisfying the $n$-th and $(n-1)$-th term are minimal and the differential between them is an admissible epimorphism in $\CE^{\rm{scw}}_{\CX}$. If $n=1$, then by the  minimal property we can observe that the differential between $0$-the term and $1$-the term is an isomorphism. Hence $F^{\bu}_2$ an zero object in  $\D^{\rm{b}}(\CS^{\rm{scw}}_{\CX}(\La))$, so the result follows. We may assume that $n>1$, and replace $A^{\bu}$ with $A^{\bu}_2$. So we can assume that $(A_{n-1}\st{f_{n-1}}\rt B_{n-1})$ and $(A_n \st{f_n}\rt B_n)$ both are minimal and $(\psi^{n-1}_1, \psi^{n-1}_2)$ an  admissible epimorphism. Since $(\psi^{n-1}_1, \psi^{n-1}_2)\circ (\psi^{n-2}_1, \psi^{n-2}_2)=0$, hence there is a unique morphism $(\phi^{n-2}, \xi^{n-2}):(A_{n-2}\st{f_{n-2}}\lrt B_{n-2}) \rt (A_{n-1}\st{f_{n-1}}\lrt B_{n-1})$ such that $(\psi_1^{n-2}, \psi_2^{n-2})=(u_{n-1}, w_{n-1})\circ (\phi^{n-2}, \xi^{n-2}).$ Now  we continue the previous  construction for the morphism $(\phi^{n-2}, \xi^{n-2})$ instead of $(\psi_1^{n-1}, \psi_2^{n-1})$. By proceeding such  constructions we finally will reach to an acyclic complex of $\CS^{\rm{scw}}_{\CX}(\La)$, that is a zero object in $\D^{\rm{b}}(\CS^{\rm{scw}}_{\CX}(\La))$. This means that we can obtain $A^{\bu}$ from the stalk complexes with the  terms of the desired forms  in the way of Construction \ref{Cons}, so the result follows.  The proof is completed. 
\end{proof}
As an application of our theorem above and in view of Example \ref{Example 4.5}, the bounded derived category $\mathbb{D}^{\rm{b}}(\mmod \Pi_n)$ of bounded complexes over the preprojective algebra $\Pi_n$ can be realized by the $\mathbb{D}^{\rm{b}}(\CS^{\rm{scw}}(\La_{n+1}))$.

Let us denote by $\overline{\D^{\rm{b}}\Psi_{\CX}}$ the induced triangle functor which gives the triangle equivalence in Theorem \ref{Theorem 6.5}.  In the next result we will show that  this induced functor provides a triangle equivalence between the singularity categories associated with  $\CS^{\rm{scw}}_{\CX}(\La)$ and $\mmod \underline{\CX}$.
\begin{theorem}\label{Theorem 6.6}
	Let $\CX$ be a resolving and  contravariantly finite subcategory of $\mmod \La.$ Then the triangle functor $\overline{\D^{\rm{b}}\Psi_{\CX}}$ induces the following triangle equivalence
	$$\D_{\rm{sg}}(\CS^{\rm{scw}}_{\CX}(\La))\simeq \D_{\rm{sg}}(\mmod \underline{\CX}).$$ 
\end{theorem}
\begin{proof}
Let $\K^{\rm{b}}(\CP(\CS^{\rm{scw}}_{\CX}(\La)))$ be the homotopy category of bounded complexes of projective objects in the exact category $\CS^{\rm{scw}}_{\CX}(\La)$. Let $\CU$ be as in Theorem \ref{Theorem 6.5}. Clearly, $\CU \subseteq \K^{\rm{b}}(\CP(\CS^{\rm{scw}}_{\CX}(\La))).$ Hence, we can identify the Verdier quotient  $\K^{\rm{b}}(\CP(\CS^{\rm{scw}}_{\CX}(\La)))/ \CU$ as a full triangulated subcategory of $\D^{\rm{b}}(\CS^{\rm{scw}}_{\CX}(\La))/\CU $. By the definition of $\D^{\rm{b}}\Psi_{\CX}$ is easy to check that the image of $\K^{\rm{b}}(\CP(\CS^{\rm{scw}}_{\CX}(\La)))/ \CU$ under the functor $\overline{\D^{\rm{b}}\Psi_{\CX}}$ is exactly $\K^{\rm{b}}(\rm{prj}\mbox{-}\underline{\CX})$, the homotopy category of  bounded  complexes of projetive functors in $\mmod \underline{\CX}$. Therefore, we have the following commutative diagram with equivalences on the rows

\[ \xymatrix{ \D^{\rm{b}}(\CS^{\rm{scw}}_{\CX}(\La))/\CU \ar[rr]^{\overline{\D^{\rm{b}}\Psi_{\CX}}} && \D^{\rm{b}}(\mmod \underline{\CX}) \\
\K^{\rm{b}}(\CP(\CS^{\rm{scw}}_{\CX}(\La)))/ \CU	\ar@{^(->}[u] \ar[rr]^{\overline{\D^{\rm{b}}\Psi_{\CX}} \ \mid} && \K^{\rm{b}}(\rm{prj}\mbox{-}\underline{\CX})\ar@{^(->}[u]}\]
Based on the above digram, we can observe the functor $\overline{\D^{\rm{b}}\Psi_{\CX}}$ induces a triangle functor between $\D_{\rm{sg}}(\CS^{\rm{scw}}_{\CX}(\La))$  and $ \D_{\rm{sg}}(\mmod \underline{\CX})$ which must be an equivalence. We are done.
\end{proof}
Indeed, the above equivalence is a singularity version of the equivalence given in \cite{E,RZ}.\\

With help of a result from \cite{MT} and the above theorem, in the next result we will reformulate in the setting of monomorphism categories  a well-known result shown independently by  Buchweitz \cite{Bu}, Happel \cite{Ha2} and Rickard \cite{R}. They independently have shown that $\D_{\rm{sg}}(\mmod \La)$, or simply $\D_{\rm{sg}}(\La)$, is triangle equivalence to  the stable category  $\underline{\rm{Gprj}}\mbox{-}\La$,  see below for the definition,  if $\La$ is a Gorenstein algebra, i.e., the injective dimension of $\La$ in both sides is finite. Let  us define the notion of Gorenstein projective object in an abelian category $\CA$. A complex $$P^\bullet:\cdots\rightarrow P^{-1}\xrightarrow{d^{-1}} P^0\xrightarrow{d^0}P^1\rightarrow \cdots$$ of  projective objects in $\CA$ is said to be \emph{totally acyclic} provided it is acyclic and the Hom complex $\Hom_{\La}(P^\bullet, Q)$ is also acyclic for any projective object $Q$ in $\CA$. An object $M$ in $\CA$ is said to be  \emph{Gorenstein projective} provided that there is a totally acyclic complex $P^\bullet$ of   projective objects over $\CA$ such that $M\cong \rm{Ker} (d^0)$, see \cite{EJ}.  We denote by $\rm{Gprj}\mbox{-}\CA$ the full subcategory of $\CA$ consisting of all Gorenstein projective objects in $\CA$. For  when $\CA=\mmod \La$, the subcategory of Gorenstein projective modules in $\mmod \La$ is shown by $\rm{Gprj}\mbox{-}\La.$

\begin{corollary}
	Assume that $\La$ is a Gorenstein algebra. Then, there is the following equivalence of triangulated categories
	$$\D_{\rm{sg}}(\CS^{\rm{scw}}_{\mmod \La}(\La))\simeq\underline{\CS^{\rm{scw}}_{\rm{Gprj\mbox{-}\La}}(\La)}.$$
\end{corollary}
\begin{proof}
	To get the equivalence in the assertion we only need to consider the following equivalences of triangulated categories
	\begin{align*}
	\D_{\rm{sg}}(\CS^{\rm{scw}}_{\mmod \La}(\La))& \simeq \D_{\rm{sg}}(\mmod (\underline{\rm{mod}}\mbox{-}\La)) \\
	& \simeq \D_{\rm{sg}}(\mmod (\underline{\rm{Gprj}}\mbox{-}\La))\\
	&\simeq \underline{\mmod (\underline{\rm{Gprj}}\mbox{-}\La)}\\
	& \simeq \underline{\CS^{\rm{scw}}_{\rm{Gprj}\mbox{-}\La}(\La)}.
	\end{align*}
The first triangle  equivalence follows from Theorem  \ref{Theorem 6.6}, the second and third triangle equivalences  from \cite[Theorem 1.4]{MT}, and the last triangle equivalence  from Theorem \ref{Theorem 5.3}.	
\end{proof}
We point out that  in the equivalence given in the above corollary, we have not considered $\CS^{\rm{scw}}_{\rm{Gprj\mbox{-}\La}}(\La)$ as the subcategory of Gorenstein projective objects in the exact category $\CS^{\rm{scw}}_{\mmod \La}(\La)$, if  the notion of a Gorenstein projective object as of abelian categories is  defined naturally for an exact category. In this case we  reasonably expect the subcategory of Gorenstein projective objects of an exact category contains projective objects. But, the category $\CS^{\rm{scw}}_{\rm{Gprj\mbox{-}\La}}(\La)$ does not include all the projective objects in the exact category $\CS^{\rm{scw}}_{\mmod \La}(\La)$, for instance those monomrphisms $0 \rt X$ with $X \in \mmod\La \setminus  \rm{Gprj}\mbox{-}\La.$ Hence $\CS^{\rm{scw}}_{\rm{Gprj\mbox{-}\La}}(\La)$ can not be viewed  as the subcategory of Gorenstein projective objects of  the exact category $\CS^{\rm{scw}}_{\mmod \La}(\La)$.\\

 We end up this section with some application for  singular equivalences of Morita type. Recently Xiao-Wu Chen and Long-Gang Sun introduced singular equivalences of Morita type \cite{CS}, see also \cite{ZZ}.
 \begin{definition}\label{MoritaSingular}
 	Let $A$ and $B$ be two Artin algebras. We say that $A$ and $B$ are {\it singularly equivalent of Morita type} if there is a pair of bimodules $({}_AM_B, {}_BN_A)$ satisfies the following conditions:
 	\begin{itemize}
 		\item [$(1)$] $M$ is finitely generated and projective as $A^{\rm{op}}$-module and as $B$-module;
 		\item [$(2)$] $N$ is finitely generated and projective as $A$-module and as $B^{\rm{op}}$-module;
 		\item [$(3)$] There is a finitely generated $A$-$A$-bimodule $X$ with finite projective dimension such that ${}_AM\otimes_BN_A\simeq {}_AA_A\oplus {}_AX_A;$
 		\item [$(4)$] There is a finitely generated $B$-$B$-bimodule $Y$ with finite projective dimension such that ${}_BN\otimes_AM_B\simeq {}_BB_B\oplus {}_BY_B.$
 	\end{itemize}
  \end{definition}
A direct consequence of the above definition and the below-mentioned fact is $-\otimes_AM_B:\D_{\rm{sg}}(A)\rt \D_{\rm{sg}}(B)$
that is an equivalence of triangulated categories with quasi-inverse
$-\otimes_BN_A:\D_{\rm{sg}}(B)\rt \D_{\rm{sg}}(A)$, see also \cite[Proposition 2.3]{ZZ}. We mention this fact that for each $Z$ in $\mmod A$ and an $A$-$B$-bimodule $W$ with  finite projective dimension as bimodule, then $Z\otimes_AW$ has finite projective dimension in $\mmod B.$  
\begin{lemma}\label{lemmalast}
Let $(A\st{f}\rt B)$ be an object in $\CS(\La)$. If either $A$ or $B$	has  finite projective dimension, then $\CE^{\rm{scw}}$-projective dimension of $(A\st{f}\rt B)$ is finite.  
\end{lemma}
\begin{proof}
	
	Set $C:=\rm{Cok} \ f.$ The short exact sequence   
	$$0 \rt A \st{f}\rt B \rt C \rt 0,$$
	thanks to \cite[Proposition 2.7]{MT}, induces  the following long exact sequence of functors in  $\mmod \underline{\rm{mod}}\mbox{-}\La,$
	
	\[\cdots\rt  (-, \underline{\Omega_{\La}(A)}) \rt (-, \underline{\Omega_{\La}(B)}) \rt (-, \underline{\Omega_{\La}(C)})\rt (-, \underline{A})\rt (-, \underline{B})\rt (-, \underline{C})\rt F\rt 0.
	 \]
	 If we assume that one of $A$ and $B$ has finite projective dimension, then by the above long exact sequence we can see  the functor $F$ has finite projective dimension in $\mmod \underline{\rm{mod}}\mbox{-}\La.$ If we take a finite projective resolution of $F$ in  $\mmod \underline{\rm{mod}}\mbox{-}\La$, then following a similar argument as we did in Lemma \ref{Lemma6.1} for the projective resolution, we can construct a finite $\CE^{\rm{scw}}_{\mmod \La}$-projective resolution for $(A\st{f}\rt B)$. But, this means that $\CE^{\rm{scw}}_{\mmod \La}$-projective dimension of $(A\st{f}\rt B)$ is finite, so the result holds.
\end{proof}

 \begin{theorem}
 	Let $A$ and $B$ be two Artin algebras. If  $A$ and $B$ are singularly equivalent of Morita type, then there is the following equivalence of triangulated categories
 	$$\D_{\rm{sg}}(\mmod \underline{\rm{mod}}\mbox{-}A)\simeq \D_{\rm{sg}}(\mmod \underline{\rm{mod}}\mbox{-}B)$$ 	 
 \end{theorem}
\begin{proof}
Assume the pair of bimodules $({}_AM_B, {}_BN_A)$ induces a singular equivalence of Morita type between $A$ and $B$.	By applying component-wise the tensor functor $-\otimes_{A}M_{B}$ on $\CS^{\rm{scw}}(A)$, we can obtain an exact functor, denote by $\CM,$ between exact categories $\CS^{\rm{scw}}(A)$ and $\CS^{\rm{scw}}(B)$. In fact, take a monomorphism $(C\st{f}\rt D)$ in $\CS^{\rm{scw}}(A)$, $\CM(C\st{f}\rt D):=(C \otimes_A M_B \st{f\otimes_A M_B}\lrt D \otimes_A M_B)$, note that since $M$ as $B$-module is projective,  so it  preserves monomrphisms. Hence $\CM$ induces a triangle functor from $\D^{\rm{b}}(\CS^{\rm{scw}}(A))$
to $\D^{\rm{b}}(\CS^{\rm{scw}}(B))$, denote by $\D^{\rm{b}}\CM$. On the other hand, since the exact functor $\CM$ keeps the structure of monorphisms being a direct sum of the monomorphisms  in the form of $(0\rt X), (X\st{1}\rt X)$, and $\Omega_{\La}(X)\rt P_X$, for some $X$ in $\mmod A$, or equivalently projective objects in $\CS^{\rm{scw}}(\La)$, then $\D^{\rm{b}}\CM$ induces a triangle functor from $\D_{\rm{sg}}(\CS^{\rm{scw}}(A))$ to 
$\D_{\rm{sg}}(\CS^{\rm{scw}}(B))$. Let us call $\D_{\rm{sg}}\CM$ the induced triangle functor between the singularity categories. To see that $\CM(\Omega_{\La}(X)\rt P_X)=(\Omega_{\La}(X)\otimes_{A} M_B\rt X\otimes_AM_B)$ with  the same structure, only use the condition $(1)$ of Definition \ref{MoritaSingular} saying that $M$ is projective as $B$-module, other cases are trivial. Similarly, one can define the triangle functor $\D_{\rm{sg}}\CN$  from $\D_{\rm{sg}}(\CS^{\rm{scw}}(B))$ to 
$\D_{\rm{sg}}(\CS^{\rm{scw}}(A))$ arising from the functor $\CN$ obtained by applying the tensor functor $-\otimes_BN_A$. Now in view of Lemma \ref{lemmalast} and conditions $(3)$ and $(4)$ of definition \ref{MoritaSingular}, one can see easily 
that the triangle functors $\D_{\rm{sg}}\CM$ and $\D_{\rm{sg}}\CN$ provide mutually inverse equivalence between $\D_{\rm{sg}}(\CS^{\rm{scw}}(A))$ and 
$\D_{\rm{sg}}(\CS^{\rm{scw}}(B))$. Now Theorem \ref{Theorem 6.6} completes the proof. 
	\end{proof}
\section{Acknowledgments}
The first named  author thanks the Institut Teknologi Bandung for providing a stimulating research environment during of his visit in ITB. This research is founded by P3MI ITB 2019.

\end{document}